\documentclass{article}
\usepackage{graphicx} 
\usepackage{amsthm}
\usepackage{amsmath}
\usepackage{amssymb}
\usepackage[margin=1.2in]{geometry}
\usepackage{hyperref}
\usepackage[shortlabels]{enumitem}
\usepackage{parskip}
\usepackage{subcaption}
\usepackage{xcolor}
\newcommand{\Blu}[1]{{#1}}

\newtheorem{thm}{Theorem}
\newtheorem{prp}{Proposition}
\newtheorem{cor}{Corollary}
\newtheorem{lem}{Lemma}[section]
\theoremstyle{definition}
\newtheorem{dfn}{Definition}
\newtheorem{rem}[lem]{Remark}
\numberwithin{equation}{section}

\newcommand{\Ss}{\mathbb{S}}
\newcommand{\RR}{\mathbb{R}}

\newcommand{\NN}{\mathbb{N}}
\newcommand{\pare}[1]{\left(#1\right)}
\newcommand{\Pare}[1]{\big(#1\big)}
\newcommand{\PAre}[1]{\Big(#1\Big)}

\newcommand{\AND}{\qquad\text{and}\qquad}
\newcommand{\AnD}{\quad\text{and}\quad}

\DeclareMathOperator{\Supp}{supp}
\DeclareMathOperator{\Conv}{conv}

\newcommand{\tS}[1]{\text{supp}_t #1}

\newcommand{\usc}{u_\text{sc}}
\newcommand{\uin}{u_\text{in}}
\newcommand{\Cor}{\mathcal{K}}
\newcommand{\wl}{w^{(l)}}

\title{Split Corners in Scattering and Inverse Scattering}
\author{Hayden Ruff \thanks{Department of Mathematics, Drexel University, Philadelphia, PA, USA. Email: hr442@drexel.edu} \and Jingni Xiao \thanks{Department of Mathematics, Drexel University, Philadelphia, PA, USA. Email: jingni.xiao@drexel.edu}}

\begin{document}
\maketitle
\begin{abstract}
	We study scattering and inverse scattering generated by sources and penetrable media with corner singularities \Blu{in $\RR^2$}. 
	We introduce the notion of split corners, a local model \Blu{based on sector decompositions} that unifies geometric singularities and coefficient discontinuities arising in source and medium scattering. Within this framework, we establish scattering and inverse scattering results that extend the classical theory to split corners, allowing the source or medium contrast to approach \Blu{possibly different essential limits} in distinct sectors meeting at the corner tip. 
	For source scattering, we prove that every admissible \Blu{$2$-sector} corner necessarily radiates in a general bounded inhomogeneous background, thereby generalizing the classical corner radiation principle to configurations involving both geometric corners and jump discontinuities. For penetrable media, we establish analogous results together with explicit compatibility conditions for incident waves of arbitrary vanishing order. As consequences, we obtain several uniqueness results in inverse source and inverse medium scattering, including recovery of polygonal convex hulls from a single far-field measurement.
\end{abstract}

\section{Introduction}
Corner singularities play an important role in scattering theory and inverse scattering. A fundamental discovery over the past decade is that geometric singularities such as corners and edges are incompatible with invisibility. In other words, under suitable assumptions, they necessarily generate nontrivial scattered waves. 
This phenomenon has led to a number of uniqueness results for recovering the geometry from limited scattering measurements.

For penetrable media, corner scattering was first established in the pioneering works of Bl{\aa}sten, P{\"a}iv{\"a}rinta and Sylvester \cite{BPS14}, who showed that incident waves cannot pass through right-angled corners without scattering. Since then, the corner scattering principle has been generalized to strictly convex corners \cite{PSV17} and more general polygonal and polyhedral corners \cite{EH18}. 
Analogous corner scattering phenomena have been established for source problems as well as electromagnetic and elastic systems \cite{LX17,Bla18,BL19,BLX21,DGT25}. Similar phenomena have also been proved for operators in divergence form, anisotropic media, and conductive transmission problems; see, for example, \cite{CX21,BL21,DCL21,DDL22,CHLX25}. These results demonstrate that geometric singularities are a robust mechanism for producing scattering across a broad range of wave models. 
More recently, it has been shown that even weaker geometric singularities lead to nontrivial scattering under suitable admissibility assumptions; see, for instance, \cite{CV23,SS21,CVX23,KSS24}. This corner scattering principle has since been extended to inverse problems on the unique determination of the convex polygonal shapes of sources and media using a single scattering measurement \cite{HSV16,HL20,CX21,BL21,DFLY24,DFL25}. 

The common feature of the above results is that the source or medium contrast approaches a single limiting value near the singularity, so that the scattering mechanism is driven entirely by the geometry of the corner.
In many applications, however, the material parameters themselves may be discontinuous near the geometric singularity. Typical examples include composite materials, piecewise-constant media, interfaces meeting at a common vertex, and source distributions with jump discontinuities. In such situations, geometric and coefficient singularities coexist and interact, and it is not clear whether the classical corner scattering mechanism remains valid. Despite their practical relevance, such configurations fall outside the scope of existing corner scattering theory, where the coefficient is typically assumed to approach a single limiting value near each corner. Consequently, the existing theory treats geometric singularities and coefficient discontinuities separately. Our notion of split corners provides a common framework encompassing both phenomena, allowing the local scattering mechanism to be analyzed without distinguishing between purely geometric and purely coefficient singularities.

To study such combined geometric and coefficient singularities, we introduce the notion of split corners, where the source or medium contrast approaches \Blu{possibly different essential limits in sectors} meeting at the corner tip. From a geometric point of view, split corners interpolate between classical geometric corners and flat interfaces carrying jump discontinuities. This class includes classical corners, flat interfaces with jump discontinuities, and mixed configurations involving both geometric and coefficient singularities. 

\Blu{Split corners provide a common local description of geometric and coefficient singularities.
Accordingly, the source and medium scattering problems are governed by a common analytical approach based on local integral identities and complex geometrical optics solutions.}
Classical corner scattering corresponds \Blu{to the special case of admissible $1$-sector (nonsplit) corners, so the split-corner framework extends the classical theory to configurations involving multiple sectors.}

We first develop this theory for source scattering. 
We prove that any admissible \Blu{$2$-sector} corner necessarily radiates. The result extends the corner-radiation principle of \cite{Bla18} to  a broader class of singular configurations that include both geometric corners and flat interfaces carrying jump discontinuities. 
The proof is based on local corner integral identities together with complex geometrical optics (CGO) solutions introduced in~\cite{Xiao_2022}, which enable an explicit asymptotic analysis near the corner. As consequences of these local results, we establish several uniqueness results for inverse source problems. Specifically, we show that exposed corners of two sources producing identical far-field patterns must satisfy strong compatibility conditions. This leads to the unique determination of the polygonal convex hull of sources from a single far-field measurement. 

We then extend the same framework to penetrable media. While the analytical approach is similar, the interaction between the incident wave and the medium introduces a richer local structure. 
We first prove that admissible \Blu{$2$-sector} corners always scatter incident waves that are nonzero at the corner tip. More generally, for incident waves of arbitrary vanishing order we derive explicit compatibility conditions relating the split-corner geometry, the \Blu{essential limits} of the contrast, and the vanishing order of the incident field. These compatibility conditions characterize the exceptional configurations under which corner scattering may fail and imply that broad classes of split corners necessarily scatter. As consequences, we obtain uniqueness results for inverse medium scattering, including recovery of polygonal convex hulls from a single far-field measurement.

The remainder of the paper is organized as follows. Section~\ref{sec:source} develops the theory for source scattering and derives its inverse consequences. Section~\ref{sec:medium} treats penetrable media, establishes the corresponding scattering results, and proves the inverse uniqueness theorems.

\section{Source Scattering}\label{sec:source}
Consider the scattering problem due to the presence of a source $f\in L^2(\RR^2)$ given by
\begin{equation}\label{eq:sourcescattering}
\begin{cases}
    \Delta u_\text{sc}+k^2\rho u_\text{sc}=f \qquad\text{in $\RR^2$},& \\
    \lim_{r\to\infty} r^{1/2}\left(\frac{\partial u_\text{sc}}{\partial r}-ik u_\text{sc}\right)=0, &
\end{cases}
\end{equation}
where $r=|x|$, $\rho\in L^\infty(\RR^2)$ \Blu{is real-valued and} signifies the background medium, and $f$ and $\rho-1$ are both compactly supported in $\RR^2$. The second equation in \eqref{eq:sourcescattering} is the Sommerfeld radiation condition, which is satisfied uniformly in all directions $\hat{x}=x/r\in\Ss^1$. Thanks to the radiation condition, the scattered field admits an asymptotic behavior at infinity, that is
\begin{equation*}
	\usc(x)=|x|^{-1/2}e^{i k|x|}\pare{u_\infty(\hat x)+O(|x|^{-1})},\qquad |x|\to\infty,
\end{equation*}
where $u_\infty$ is referred to as the far-field pattern. Moreover, the source scattering problem \eqref{eq:sourcescattering} admits a unique solution $\usc\in H^1_\text{loc}(\RR^2)$. 

In inverse scattering, one measures the far-field pattern $u_\infty$, and \Blu{seeks to recover information about} the source function $f$. However, there is a class of so-called ``non-radiating'' sources which produce trivial far-field $u_\infty\equiv 0$. 
Thus an exterior observer will not be able to ``see'' anything as if the source $f$ did not exist. A typical example of non-radiating source is constructed by considering a solution $\phi$ to $\Delta \phi +k^2\rho\phi =0$ in $\RR^2$ and a cutoff function $\chi\in C_c^{2}(\RR^2)$, and let $f:=(\Delta  +k^2\rho)(\chi\phi)$. Then $\Supp f\subseteq\Supp\chi$ and $\usc=\chi\phi$ is the unique solution to \eqref{eq:sourcescattering}, which has a trivial (i.e., identically zero) far-field. Notice in this case that $f=0$ on $\partial \Supp f$ generally. 
The example above shows that non-radiating sources may exist when the source vanishes at the boundary of its support. In contrast, existing corner-radiation results indicate that nontrivial boundary values at geometric corners prevent such invisibility.

Note that $\Supp f$ generally \Blu{has} ``holes'' where $f$ is zero. The following definition of ``total support'' removes such holes from the support; see also, the notion of infinity support in \cite{KusSyl03}.
\begin{dfn}\label{dfn:scattersupport}
	Given $f\in L^2(\RR^2)$ with compact support, let $G$ be the unbounded connected \Blu{component} of $\RR^2\backslash\Supp f$. The \textit{total support} of $f$ is defined as $\tS f:= \RR^2\backslash G$.
\end{dfn}
\begin{figure}[h]
\begin{subfigure}{0.5\textwidth}
\centering\includegraphics[width=0.4\linewidth]{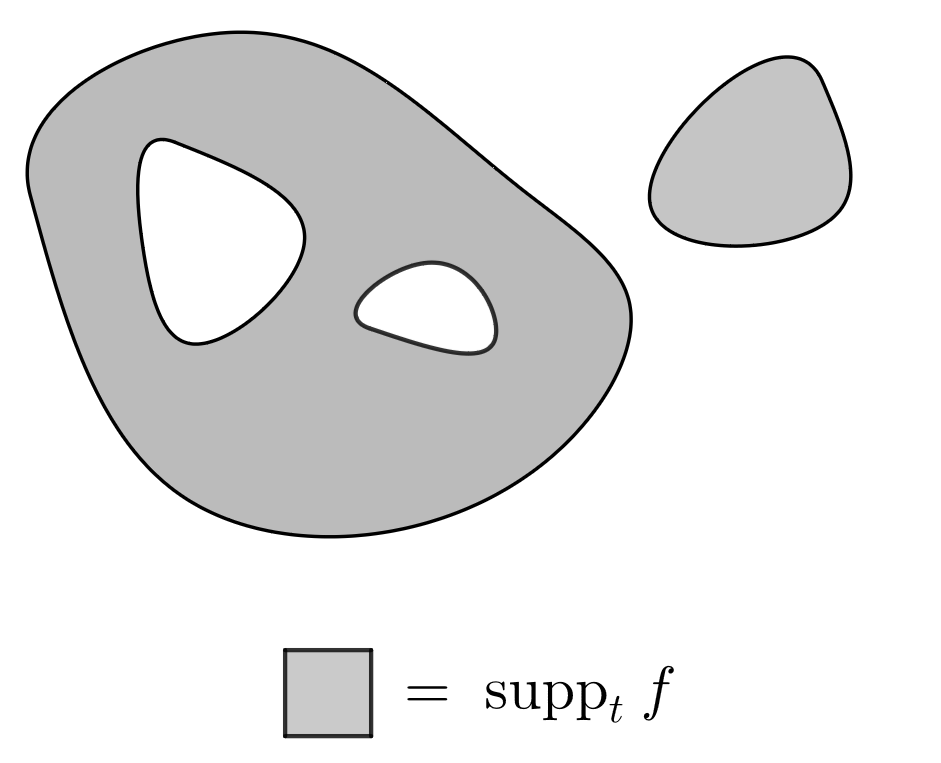} 
\label{fig:suppf}
\end{subfigure}
\begin{subfigure}{0.5\textwidth}
\centering\includegraphics[width=0.4\linewidth]{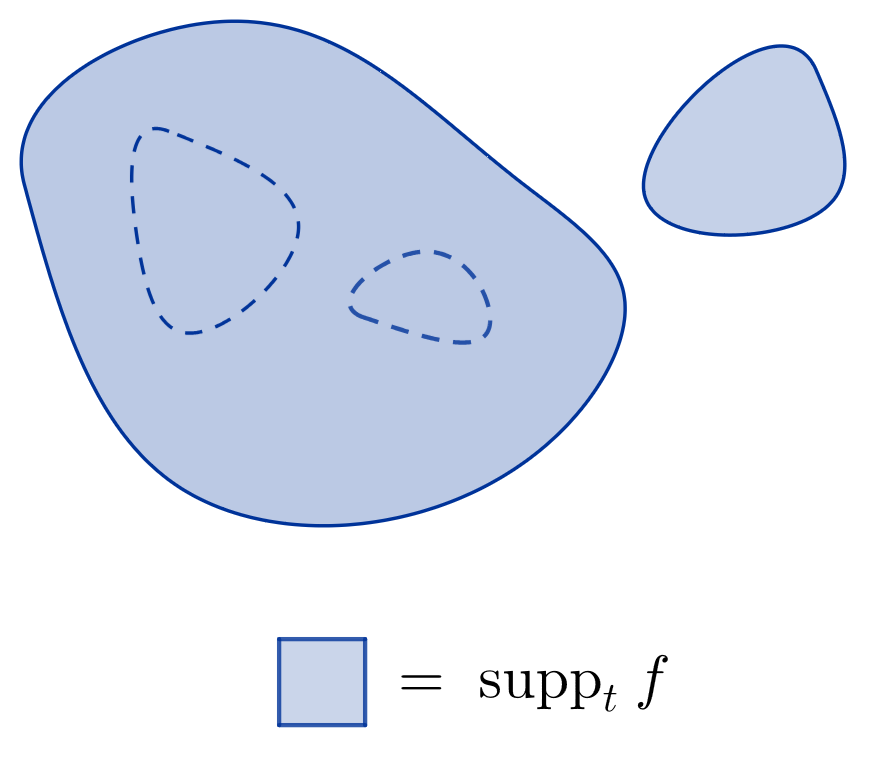}
\label{fig:tsuppf}
\end{subfigure}
\caption{A visual comparison of supp$f$ and $\text{supp}_t f$. }
\label{fig:supptsupp}
\end{figure}

In this paper, by a \textit{source} we refer to a compactly supported $f\in L^2(\RR^2)$, or the pair $(f,\Sigma):=(f,\tS f)$.  
Given a compact set $\Sigma$, we denote \Blu{by $\Sigma^\circ$} its interior.
We also define the circular wedge $\Cor _{\theta_0,\theta_1, R}:=\{x\in \RR^2:\;\theta_0<\theta<\theta_1,0<r<R\}$ with $R\in\RR_+$ and $\theta_0,\theta_1\in [-\pi,\pi)$ such that $0<\theta_1-\theta_0<2 \pi$.

We are particularly interested in understanding what happens when the boundary singularity is no longer purely geometric, but is accompanied by discontinuities of the source itself. This motivates the notion of a split corner introduced below.
The purpose of introducing split corners is to describe singularities arising from two independent mechanisms. The first is the geometric singularity created by a non-flat corner, while the second is a discontinuity in the \Blu{essential limit} of the source or medium across an interface.

Either mechanism can generate scattering, whereas configurations lacking both mechanisms behave locally like smooth interfaces and \Blu{are therefore not expected to satisfy a corner scattering principle. This motivates the admissibility conditions introduced below, which isolate configurations possessing at least one detectable singular feature.}
\begin{dfn}\label{dfn:splitcorner}
	Given a domain $D$ in $\RR^2$, we say that $D$ has a \textit{corner} $\Cor$ if, up to a rigid change of coordinates, $0\in\partial D$ and $B_R(0)\cap D=\Cor=\Cor _{\theta_0,\theta_1,R}$ for some $R>0$. The \Blu{\textit{opening angle}	of $\Cor$} is $\omega:=\theta_1-\theta_0$. We say that a source $(f,\Sigma)$ has a \textit{corner} or is \textit{cornered} if $\Sigma^\circ$ has a corner\footnote{\Blu{In this paper, we generally allow an opening angle to be $\omega=\pi$, so that a ``corner'' may be geometrically flat.}}.
	
	\Blu{Suppose that $f$ has a corner $\Cor=\Cor_{\theta_0,\theta_n,R}$.
	We say that $f$ has an \textit{$n$-sector corner} at $\Cor$, or that $\Cor$ is an
	\textit{$n$-sector corner} of $f$, if there exist angles $\theta_0<\theta_1<\cdots<\theta_n$ and real numbers $p_j$\footnote{For simplicity of presentation, we assume that the source term $f$ is real-valued. All results remain valid for complex-valued $f$, since equation \eqref{eq:sourcescattering} (and almost all other equations in this article) can be separated into real and imaginary parts.}, together with constants $C_j,\alpha_j>0$, such that
	\[
	|f(x)-p_j|\leq C_j|x|^{\alpha_j}
	\quad\text{for a.e. }x\in\Cor_j,
	\qquad j=1,\ldots,n,
	\]
	where $\Cor_j:=\Cor_{\theta_{j-1},\theta_j,R}$, $j=1,\ldots,n$.
	We call $\omega_j:=\theta_j-\theta_{j-1}$ the \textit{\Blu{opening angle}} of the $j$-th sector and $(p_1,\ldots,p_n)$ the \textit{\Blu{essential limits}} of $f$ at $\Cor$.
	We use the term \textit{split corner} collectively for this class of corners, and call the case $n=1$ a \textit{nonsplit corner}.}
	
	\Blu{A $1$-sector corner is said to be \textit{admissible} if $p_1\neq0$ and $\omega\neq\pi$.
	A $2$-sector corner is \textit{admissible} if one of the following holds:
	\begin{enumerate}[(i)]
		\item\label{cs:pitot}
		$\omega=\pi$ and $p_1\neq p_2$;
		\item\label{cs:fn0}
		$\omega\neq\pi$ and one of the following holds:
		\begin{enumerate}[(a)]
			\item $\omega_1=\pi$ and $p_2\neq0$;
			\item $\omega_2=\pi$ and $p_1\neq0$;
			\item $\omega_1,\omega_2\neq\pi$ and
			$(p_1,p_2)\neq(0,0)$.
		\end{enumerate}
	\end{enumerate}}
\end{dfn}
\begin{rem}
	\Blu{Note that the essential limits in adjacent sectors need not be distinct. Thus, the term \textit{split corner} refers to the sector decomposition and does not necessarily imply a jump discontinuity. In particular, a $2$-sector corner with $p_1=p_2$ can also be viewed as a nonsplit corner.}
	
	\Blu{The admissibility conditions distinguish three types of geometric configurations.
	When $\omega=\pi$, admissibility requires $p_1\neq p_2$, corresponding to a jump across the interior interface.
	When $\omega_1=\pi$ or $\omega_2=\pi$,  the essential limit in the other sector 	must be nonzero. In the remaining case, where none of $\omega,\omega_1,\omega_2$ equals $\pi$, at least one essential limit must be nonzero.}
	
	\Blu{These conditions arise naturally in the proof of Proposition~\ref{prp:source2}.
	The leading-order asymptotics yield a homogeneous linear system for $(p_1,p_2)$ whose determinant is proportional to $\sin\omega\sin\omega_1\sin\omega_2$.
	Equivalently, admissibility requires that the leading-order system cannot be satisfied by the prescribed essential limits.} See Figure~\ref{fig:admiss2split} for representative examples.
\end{rem}
\begin{figure}[h]
	\begin{subfigure}{0.24\textwidth}
		\centering\includegraphics[scale=.3]{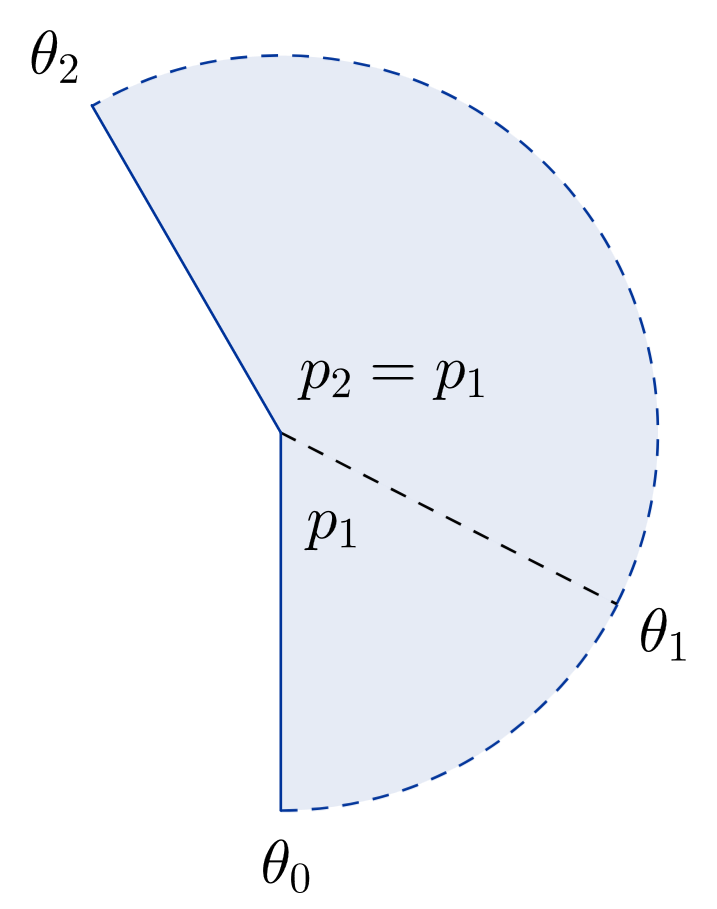}
		\caption{$\omega\neq \pi$, $p_1=p_2\neq 0$. }
	\end{subfigure}
	\begin{subfigure}{0.24\textwidth}
		\centering\includegraphics[scale=.3]{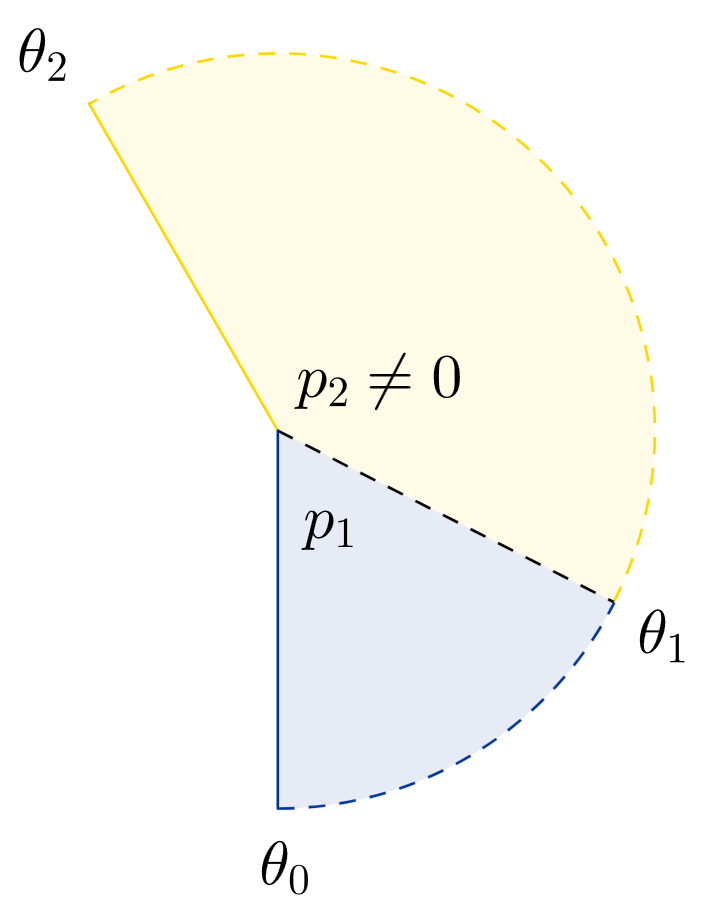}
		\caption{$\omega_2\neq \pi$, $p_2\neq 0$. }
		\label{fig:noflat}
	\end{subfigure}
	\begin{subfigure}{0.24\textwidth}
		\centering\includegraphics[scale=.3]{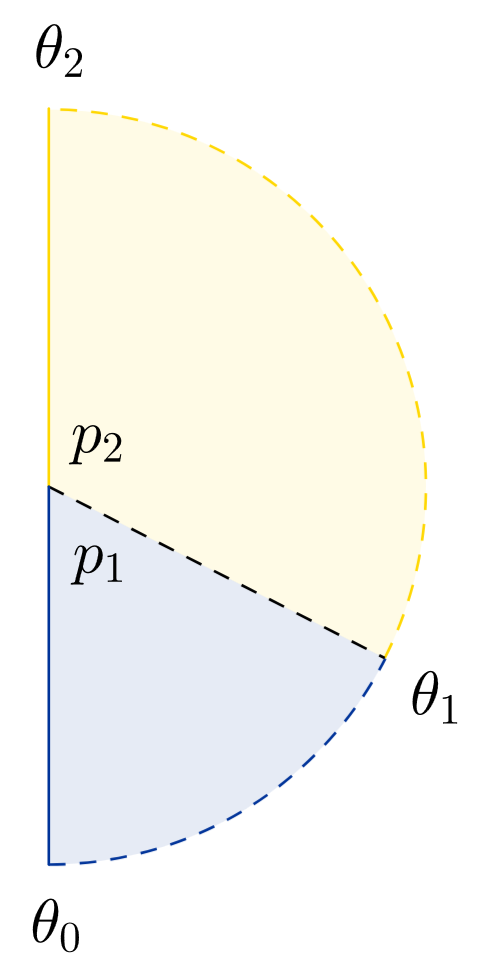} 
		\caption{$\omega=\pi$, $p_1\neq p_2$.}
		\label{fig:t0t2flat}
	\end{subfigure}
	\begin{subfigure}{0.24\textwidth}
		\centering\includegraphics[scale=.3]{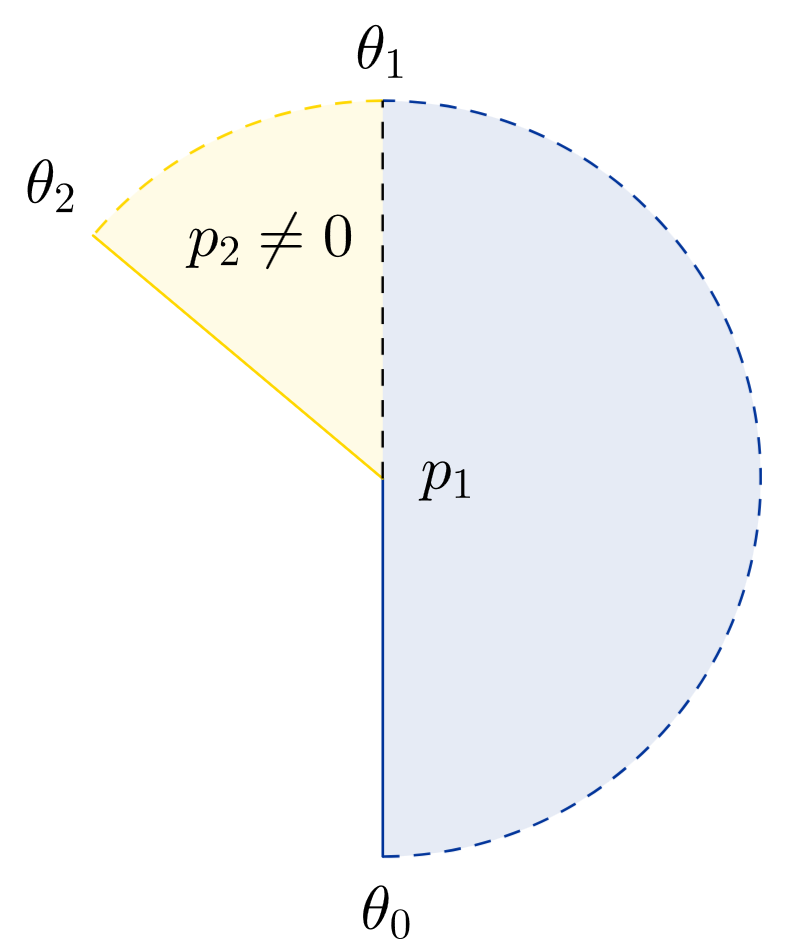}
		\caption{$\omega_1=\pi$, $p_2\neq 0$. }
		\label{fig:t0t1flat}
	\end{subfigure}
	\caption{Examples of admissible $2$-sector corners.} 
\label{fig:admiss2split}
\end{figure}

Although Definition~\ref{dfn:splitcorner} allows \Blu{$n$-sector corners for arbitrary $n\geq1$, the $2$-sector} case already captures the interaction between geometric and coefficient singularities while keeping the resulting compatibility conditions explicit. We therefore focus on \Blu{the $1$- and $2$-sector cases in this paper}.
\subsection{Main Results}
We say that a source (always) radiates (nontrivially) if the corresponding far-field $u_\infty\not\equiv 0$. 

The result below shows that any source with an admissible \Blu{$2$-sector} corner always radiates:
\begin{thm}[Corners always radiate]\label{thm:splitscatter}
  Any source possessing an admissible \Blu{$2$-sector} corner always radiates, even in a general bounded inhomogeneous background.
\end{thm}
\begin{rem}
	Theorem~\ref{thm:splitscatter} unifies two corner-radiation principles: non-flat corners with nonvanishing \Blu{essential limits} always radiate, and flat interfaces carrying jump discontinuities always radiate.
	Admissible split corners interpolate between these two situations.
\end{rem}
\Blu{The following corollary recovers the corner-radiation result for admissible $1$-sector corners, established, for example, in \cite{Bla18}. }
\begin{cor}\label{cor:source1}
	  \Blu{A source possessing an admissible $1$-sector corner always radiates.}
\end{cor}

The following local non-radiation criterion is the key technical result of this section. It immediately implies Theorem~\ref{thm:splitscatter} and serves as the foundation for the inverse uniqueness results proved below. We postpone its proof until Section~\ref{sec:proof}.
\begin{prp}\label{prp:source2}
	Suppose there exists $u\in H^1(\Cor )$ \Blu{with $\Cor=\Cor_{\theta_0,\theta_1,R}$} satisfying\footnote{\Blu{The problem~\eqref{eq:sourcelocprob}, as well as other similar local problems in the rest of the paper, is understood in the distributional sense. The normal derivative on the boundary is interpreted as weak traces. In particular, if $u\in H^1(\Cor)$ satisfies $(\Delta+k^2\rho)u=f$ with $\rho\in L^\infty(\Cor)$ and $f\in L^2(\Cor)$, then $\Delta u\in L^2(\Cor)$ and its weak normal derivative is well-defined.}}
	\begin{align}\label{eq:sourcelocprob}
		\begin{cases}
			(\Delta+k^2\rho)u=f &\quad\text{in $\Cor $}, \\
			u=0 \AnD \partial_\nu u=0 &\quad\text{on $\partial\Cor \cap B_R$},
		\end{cases}
	\end{align}
with $\rho\in L^\infty(\Cor)$. 
	 Then $\Cor $ is not an admissible \Blu{$2$-sector} corner of $f$. 
\end{prp}

Proposition~\ref{prp:source2} can be applied to inverse problems of identifying the shape and location $\Sigma$ of a source $f$ using a single measurement. 
In the rest, we assume for any source $(f,\Sigma)$ that $\overline{\Sigma^\circ}=\Sigma$. 
Our first inverse result concerns two sources producing identical far-field patterns. 
It shows that any exposed corner connected to infinity cannot be an admissible \Blu{$2$-sector} corner. The result holds in a general known inhomogeneous background satisfying $\rho\in L^\infty(\mathbb R^2)$ with compactly supported perturbation $\rho-1$.
\begin{prp}\label{prp:invSourceGeneral}
	Suppose that $(f^{(j)},\Sigma_j)$, $j=1,2$, are two sources producing the same far-field. Then for each $j=1,2$, $(f^{(j)},\Sigma_j)$ does not have any admissible \Blu{$2$-sector} corner on $\Gamma_j$, where $\Gamma_j$ is any connected and relatively-open subset of $\partial\Sigma_j\backslash\Sigma_{j+1}$ that connects to infinity in $\RR^2\backslash(\Sigma_1\cup \Sigma_2)$, with $\Sigma_3:=\Sigma_1$.
\end{prp}
\begin{figure}[h]
\begin{subfigure}{0.5\textwidth}
\centering\includegraphics[width=0.6\linewidth]{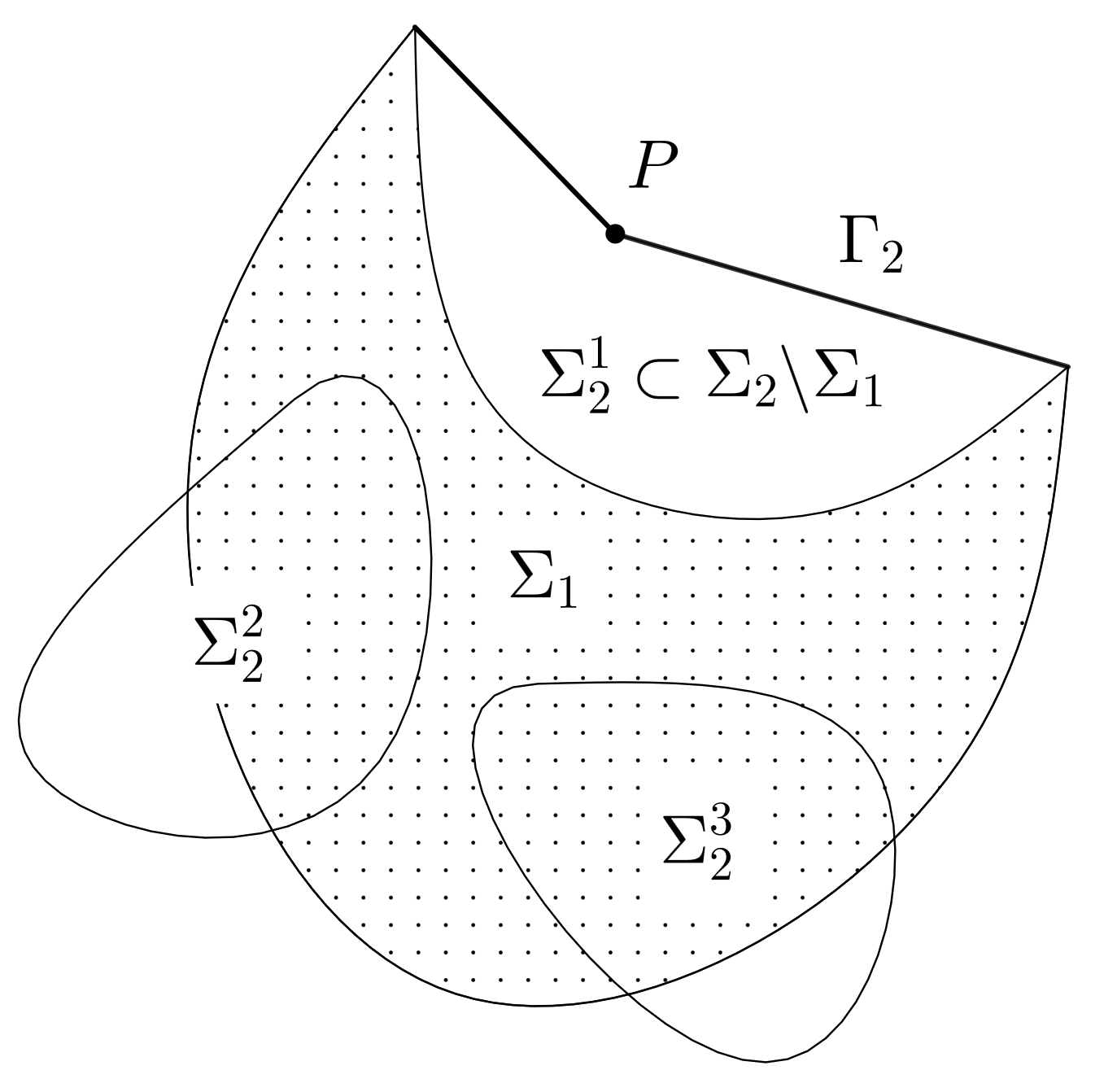} 
\caption{$\Sigma_2=\bigcup_{j=1}^3\Sigma_2^j$ union of disjoint sets.}
\label{fig:disjunion}
\end{subfigure}
\begin{subfigure}{0.5\textwidth}
\centering\includegraphics[width=0.6\linewidth]{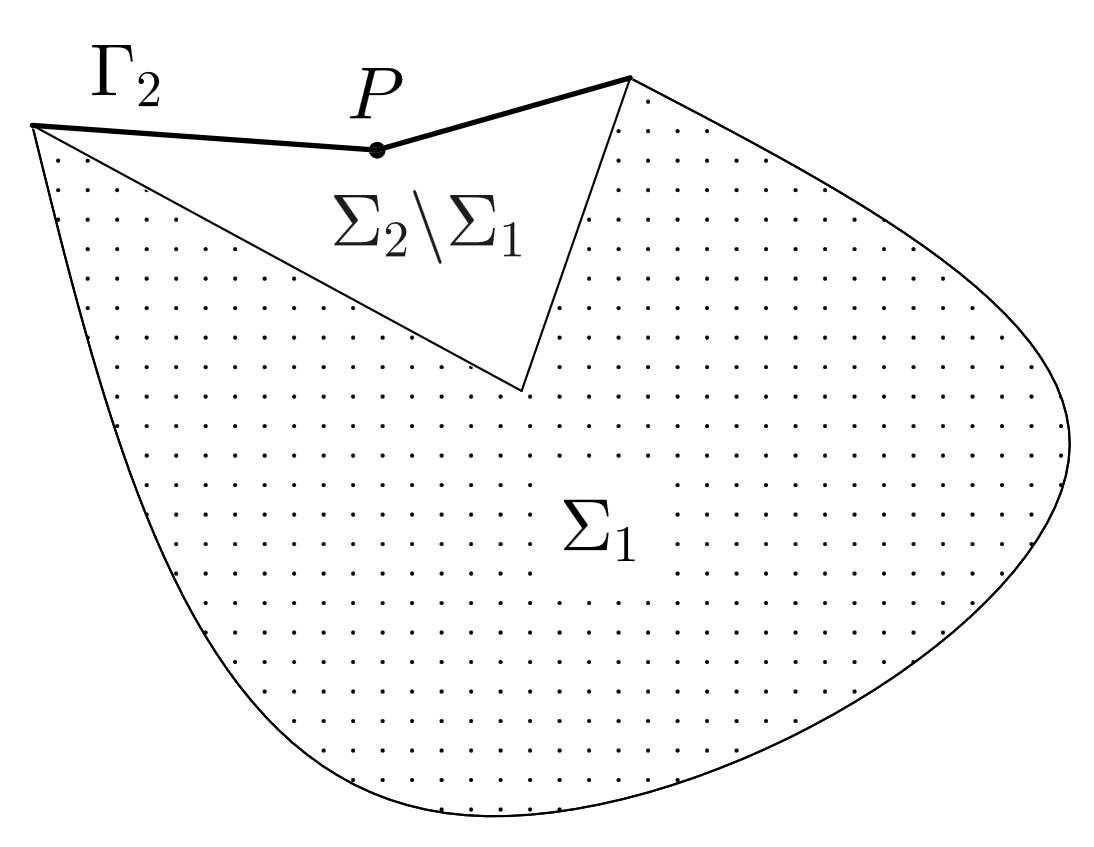}
\caption{$\Sigma_2$ simply connected, conv$(\Sigma_1)=$conv $(\Sigma_2)$ }
\label{fig:simpconnect}
\end{subfigure}
\caption{In each image, if $(f_2, \Sigma_2)$ has \Blu{an admissible $1$- or $2$-sector} corner at the point $P$, Proposition \ref{prp:invSourceGeneral} guarantees that $u_{\infty,1}\neq u_{\infty,2}$.}
\label{fig:equalffex}
\end{figure}

We next consider polygonal sources, namely sources $(f,\Sigma)$ for which $\Sigma$ is a union of finitely many disjoint (possibly concave) polygons.
\begin{thm}\label{thm:inversescattering}
	Let $(f,\Sigma)$ be a polygonal source such that each corner of $\Sigma$ is \Blu{an admissible $1$- or $2$-sector corner of} $f$. Then $\Conv (\Sigma)$ can be uniquely recovered from the corresponding far-field for a fixed wavenumber $k$.
\end{thm}
\begin{rem}\label{rem:outerpolyg}
	Although Theorem~\ref{thm:inversescattering} is stated for polygonal supports, the proof only uses the corner geometry. Hence the same conclusion holds whenever $\Sigma$ is contained in a polygonal region $\Sigma'$ sharing the same corners, so that $\operatorname{conv}(\Sigma)=\operatorname{conv}(\Sigma')$. See Figure~\ref{fig:innerregouterpolyg}.
\end{rem}
\begin{figure}[h!]
	\centering \includegraphics[scale=.35]{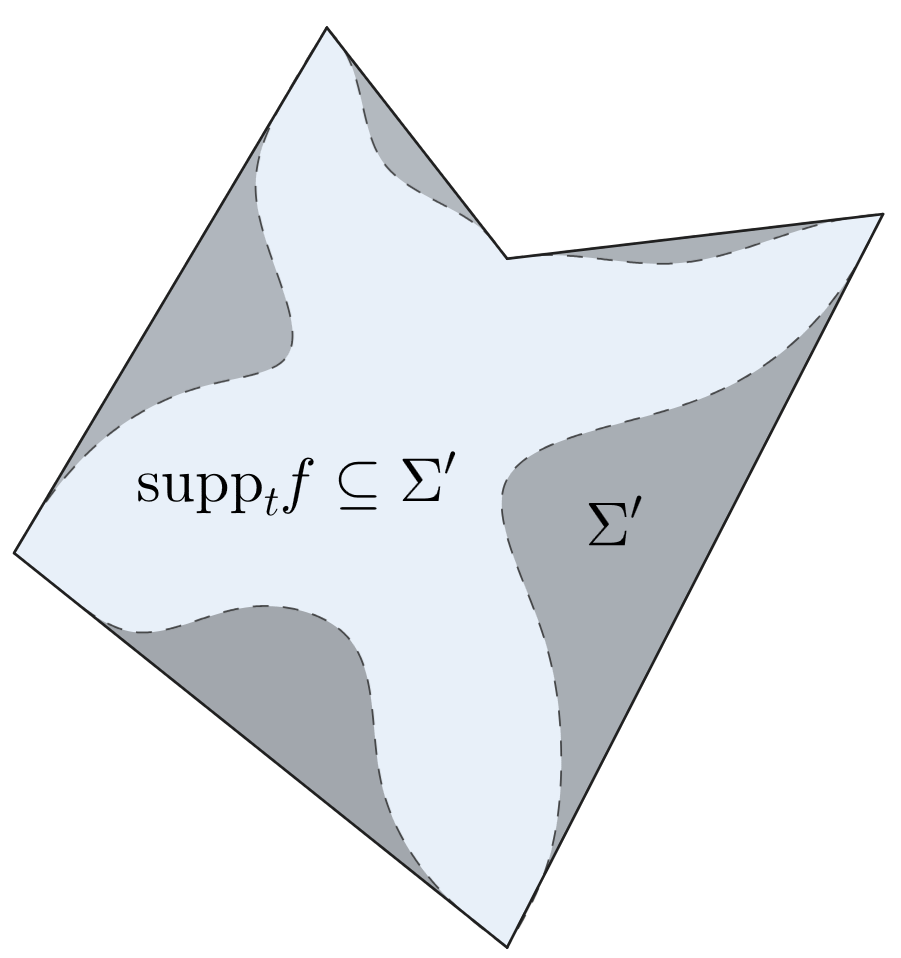}
	\caption{An example of a region $\tS{f}$ contained in a polygonal region $\Sigma'$ sharing its corners, as described in Remark \ref{rem:outerpolyg}.}\label{fig:innerregouterpolyg}
\end{figure}

For polygonal sources, the previous uniqueness result can be strengthened by deriving compatibility conditions for exposed vertices of two sources producing identical far-field patterns.
In particular, all ``exposed vertices'' must satisfy strong geometric compatibility conditions.
\begin{prp}\label{prp:inversescattering1}
	For \Blu{$j=1,2$, let $(f^{(j)},\Sigma_j)$ be polygonal sources such that every boundary point of $\Sigma_j$ is locally a $1$-sector corner of $f^{(j)}$, possibly with opening angle $\pi$, with a nonzero essential limit in the sense of Definition~\ref{dfn:splitcorner}.
	Suppose that the two sources produce the same far-field pattern, and let $G$ be the unbounded connected component of $\RR^2\setminus(\Sigma_1\cup\Sigma_2)$.
	Then every vertex $x'\in\partial G$ of $\Sigma_1$ (resp., $\Sigma_2$) is also a vertex of $\Sigma_2$ (resp., $\Sigma_1$).
	Moreover,
	\begin{equation}\label{eq:inversecornercompat}
		p_1e^{2i\gamma_1}\sin\omega_1=p_2e^{2i\gamma_2}\sin\omega_2,
	\end{equation}
	where $p_j\neq0$ is the essential limit of $f^{(j)}$ at $x'$ from within $\Sigma_j$, and $\omega_j$ and $\gamma_j$ are the opening angle and bisector angle of $\Sigma_j$ at $x'$, respectively.
	In particular,
	\begin{equation*}
		\gamma_1=\gamma_2+m\pi/2\AND p_1\sin \omega_1=(-1)^{m}p_2\sin \omega_2
	\end{equation*}
	for some $m\in\mathbb Z$.}
\end{prp}
\begin{rem}
	If we assume a priori that the \Blu{essential limit of an unknown} polygonal source at every vertex is a fixed constant, Proposition~\ref{prp:inversescattering1} then provides necessary geometric conditions on all ``exposed corners'' of two geometries $\Sigma_1$ and $\Sigma_2$ that yield the same far-field. For example, at any vertex of $\Conv\Sigma_1=\Conv\Sigma_2$ we must have
	\begin{equation*}
		\Blu{\gamma_1=\gamma_2},\AND\mbox{either\quad \Blu{$\omega_1=\omega_2$ \quad or\quad $\omega_1+\omega_2=\pi$}}.
	\end{equation*}
	See Figure~\ref{fig:counterex} for an illustration of such situation.
\begin{figure}[h]
	\begin{subfigure}{.5\linewidth}
		\centering\includegraphics[scale=.4]{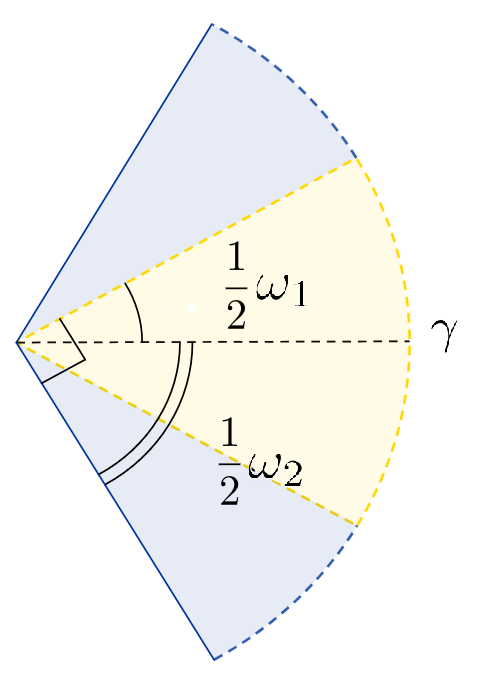} 
		\caption{Local configuration: $\gamma_1=\gamma_2$, \Blu{$\omega_1+\omega_2=\pi$}}
	\end{subfigure}
	\begin{subfigure}{0.5\textwidth}
		\centering
		{\includegraphics[scale=1.8]{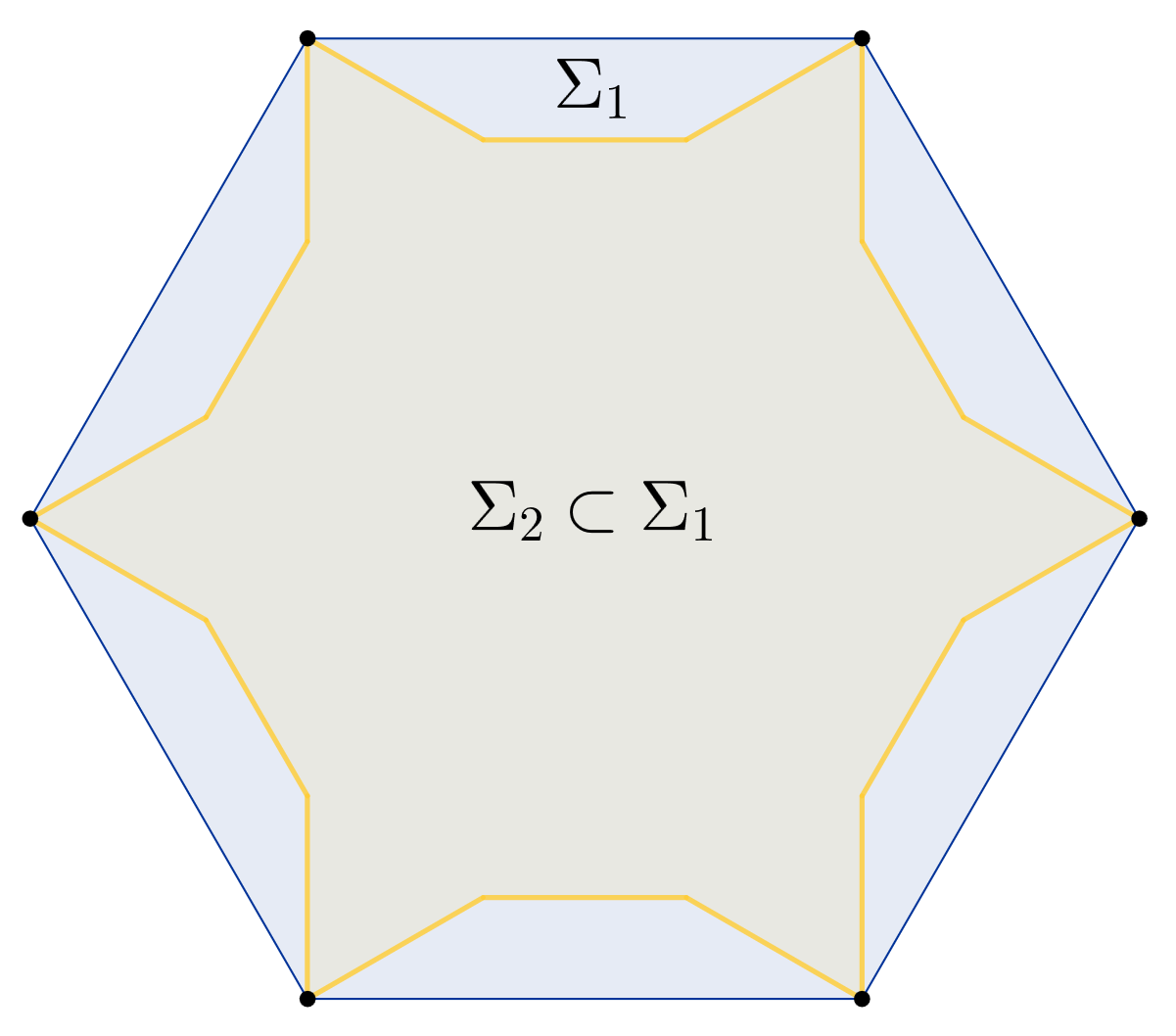}} 
		\caption{Global configuration: two geometries $\Sigma_1,\Sigma_2$}
	\end{subfigure}
\caption{Geometric configurations of two sources that cannot be distinguished by results in this section.}
\label{fig:counterex}
\end{figure}
\end{rem}

\subsection{Proof of Proposition~\ref{prp:source2}}\label{sec:proof}
To prove Proposition~\ref{prp:source2} we derive an integral identity near the corner and analyze it using complex geometrical optics solutions. By extracting the leading order asymptotics of the resulting oscillatory integrals, we obtain algebraic constraints on the \Blu{essential limits} of the source. These constraints are incompatible with the admissibility assumptions and therefore rule out non-radiating configurations.

Let $(f,\Sigma)$ be a source with a corner $\Cor _{\theta_0,\theta_2,R}$, and suppose there exists a solution $u$ to the local problem \eqref{eq:sourcelocprob} for $\Cor :=\Cor _{\theta_0,\theta_2,R}$. We obtain for any $w\in H^1(\Cor _{\theta_0,\theta_2,R})$ satisfying $\Delta w+k^2\rho w=0$ in \Blu{$\Cor _{\theta_0,\theta_2,R}$} that
\begin{equation}\label{eq:fwint}
    \begin{split}
    	\int_{\Cor _{\theta_0,\theta_2,R}} fw\;dx
    &=\int_{\Cor _{\theta_0,\theta_2,R}}-\nabla u\cdot\nabla w+k^2\rho u w\;dx+\int_{\partial \Cor _{\theta_0,\theta_2,R}}w\partial_\nu u \;dS \\
    &=\int_{\partial \Cor _{\theta_0,\theta_2,R}\cap\partial B_R}w\partial_\nu u -u\partial_\nu w \;dS.
    \end{split}
\end{equation}
\Blu{To analyze \eqref{eq:fwint}, we use the complex geometrical optics (CGO) solutions to $\Delta w+k^2\rho w=0$ given} in \cite[Proposition 2]{Xiao_2022}. We will apply these solutions to the asymptotic analysis of \eqref{eq:fwint} to determine information about $f$ at the vertex of the corner.
\begin{lem}\label{lem:CGO}
    Let \Blu{$D= \Cor _{\theta_0,\theta_1,R}$} with $\theta_0,\theta_1\in[-\pi,\pi)$ and $R\in\RR^+$. Given $q\in L^\infty(D)$ and $\beta\in (0,1)$, there is a constant $h_0\in (0,1)$ such that for each $h\in (0,h_0)$, there is a solution $\wl $ to $\Delta w+q w=0$ in $D$ of the form
\begin{align}\label{eq:CGOsoln}
    \wl(x):=e^{-\frac{1}{h} \Phi_l(x)}(1+w_h),
\end{align}
where $\Phi_l(x):=r^\beta e^{(-1)^{l-1}i\beta \theta}$ for $l=1,2$. Moreover, for any $p_1\in (1,\infty)$ and $p_2\in (1, 2/\beta)$, we have the estimates $\|w_h\|_{L^{p_1}(D)}\leq Ch$ and $\|w_h\|_{W^{1,p_2}(D)}\leq C$, where $C>0$ is independent of $h$.
\end{lem}
In the following, we shall always take $\beta\in(0,1/2)$ when applying Lemma~\ref{lem:CGO}.
Note that $\Re \Phi_l(x)=|x|^\beta \cos(\beta \theta)$. Then there exists a constant $\delta>0$ such that
\begin{align}\label{eq:delineq}
    \Re\Phi_l(x)\geq \delta |x|^\beta,
\end{align}
for all $x\in \overline{\Cor }_{\theta_0,\theta_1,R}$ whenever $-\pi\le\theta_{0}<\theta_1<\pi$.
In order to prove Proposition~\ref{prp:source2}, we will analyze the asymptotic behavior of the identity \eqref{eq:fwint} as $h\to 0$, with $w$ given in Lemma~\ref{lem:CGO}.

The following result concerning the asymptotics of the boundary integral in \eqref{eq:fwint} is standard by using CGO solutions; see, for example \cite[Lemma 4.1]{Xiao_2022}. 
\begin{lem}\label{lem:bterms}
      Let \Blu{$u\in H^1(\Cor_{\theta_0,\theta_1,2R})$ satisfy \eqref{eq:sourcelocprob} with $\Cor=\Cor_{\theta_0,\theta_1,2R}$ (so that $u\in H^2$ locally near
      	$\partial B_R\cap\overline{\Cor}$), and let $\wl $ be a solution to $\Delta w+k^2\rho w=0$} of the form \eqref{eq:CGOsoln} given by Lemma \ref{lem:CGO} in \Blu{$D=\Cor_{\theta_{0}-\epsilon,\theta_{1}+\epsilon,2R}$} for some $\epsilon>0$, where \eqref{eq:delineq} is satisfied in $D$. Then
    \begin{align*}
        \int_{\partial \Cor _{\theta_0,\theta_1,R}\cap\partial B_R}\wl \partial_\nu u -u\partial_\nu \wl  \;dS=\Blu{O(h^{-1}e^{-\frac{1}{2h} R^\beta \delta})},\qquad\mbox{as $h\to 0$}.
    \end{align*}
\end{lem}
\Blu{To analyze the asymptotic behavior of the left-hand side of \eqref{eq:fwint}, we first consider the contribution from a single sector.}
\begin{lem}\label{lem:Asymp}
 If $f$ has a $1$-sector corner \Blu{$\Cor_{\theta_0,\theta_1,R}$, then}, for $w^{(l)}$ as in Lemma~\ref{lem:bterms},
	\begin{equation*}
		\int_{\Cor_{\theta_{0},\theta_{1}},\varepsilon}f\wl dx=(-1)^{l-1}i\frac{\Gamma(2/\beta)}{2\beta } \Pare{e^{(-1)^l2i\theta_1}-e^{(-1)^l2i\theta_0}} \Blu{p}h^{2/\beta} +o(h^{2/\beta})\qquad \mbox{as $h\to 0$},
	\end{equation*}
with some $\varepsilon\in\RR_+$ independent \Blu{of $h$, where $p$ is the essential limit of $f$ at $\Cor_{\theta_0,\theta_1,R}$.}
\end{lem}
For the proof of Lemma~\ref{lem:Asymp}, \Blu{we use the following estimate; see also
	\cite{Xiao_2022}.}
\begin{lem}\label{gammest}
	Given $\varepsilon, b_0>0$ and $\mu,b_1\in\mathbb{C}$ such that $\Re{\mu},\Re{b_1}>0$, we have
	\begin{align*}
		\left|\frac{\Gamma(b_1/b_0)}{b_0}\mu^{-b_1/b_0}-\int_0^\varepsilon t^{b_1-1}e^{-\mu t^{b_0}}\;dt\right|&\leq \frac{\Gamma(\Re{b_1}/b_0)}{b_0}(2/\Re\mu)^{\Re{b_1}/b_0}e^{-(\Re{\mu}) \varepsilon^{b_0}/2}.
	\end{align*}
\end{lem}

\begin{proof}[Proof of Lemma~\ref{lem:Asymp}]
	Note for $l=1,2$ that
	\begin{equation}\label{eq:leading}
		\begin{split}
		&\int_{\Cor _{\theta_0,\theta_1,\varepsilon}}e^{-\frac{1}{h} \Phi_l(x)}dx
		=\Blu{\int_{\Cor _{\theta_0,\theta_1,\varepsilon}}e^{-\frac{1}{h} |x|^\beta e^{(-1)^{l-1}i\beta\theta}}\;dx
		=\int_{\theta_0}^{\theta_1}}\int_0^\varepsilon r e^{-\frac{1}{h} r^\beta e^{(-1)^{l-1}i\beta\theta} }\;drd\theta \\
		&\qquad=\frac{\Gamma(2/\beta)}{\beta }h^{2/\beta}\int_{\theta_0}^{\theta_1} e^{(-1)^li2\theta} d\theta +O(h^{\alpha_1} e^{-\delta \varepsilon^\beta/2h}) \\
		&\qquad=(-1)^{l-1}i\frac{\Gamma(2/\beta)}{2\beta } \pare{e^{(-1)^l2i\theta_1}-e^{(-1)^l2i\theta_0}} h^{2/\beta} +O(h^{\alpha_1} e^{-\delta \varepsilon^\beta/2h}),
	\end{split}
	\end{equation}
	where we have applied Lemma~\ref{gammest}, and $\alpha_1$ is a finite power independent of $h$.
	Similarly, we can derive that
	\begin{equation}\label{eq:Estr^a}
		\int_{\Cor _{\theta_0,\theta_1,\varepsilon}}|x|^\alpha e^{-\frac{1}{h} |x|^\beta e^{(-1)^{l-1}i\beta\theta}}\;dx
		\le C h^{(2+\alpha)/\beta}.
	\end{equation}
	Applying Lemmas \ref{lem:CGO} and \ref{gammest} we have for $\lambda\in(0,1)$ that
	\begin{equation}\label{eq:whterm}
		\begin{split}
			&\left|\int_{\Cor _{\theta_0,\theta_1,\varepsilon}} fw_h e^{-\Phi_l/h}\;dx \right|
		\leq \|f\|_{L^\infty}\int_{\Cor _{\theta_0,\theta_1,\varepsilon}} |w_h| e^{-|x|^\beta\cos(\beta\theta)/h}\;dx   \\
		&\quad\quad\quad\leq \|f\|_{L^\infty}\|w_h\|_{L^{\frac{1}{1-\lambda}}(\Cor _{\theta_0,\theta_1,\varepsilon})}\PAre{\int_{\Cor _{\theta_0,\theta_1,\varepsilon}} e^{-\delta |x|^\beta/(\lambda h)}\;dx}^{\lambda} \\
		&\quad\quad\quad\leq C h\|f\|_{L^\infty({\Cor _{\theta_0,\theta_1,\varepsilon}})}(\theta_1-\theta_0)^{\lambda}\left(\int_0^\varepsilon re^{-\delta r^\beta/(\lambda h)}\;dr\right)^{\lambda}
		\leq C_1 h^{1+2\lambda/\beta}.
		\end{split}
	\end{equation}
	Therefore,
	\begin{equation*}
		\begin{split}
			\int_{\Cor_{\theta_{0},\theta_{1}},\varepsilon}f\wl\;dx
			&=\int_{\Cor _{\theta_0,\theta_1,\varepsilon}}  p e^{-\Phi_l/h}\;dx
			+\int_{\Cor _{\theta_0,\theta_1,\varepsilon}} (f-p) e^{-\Phi_l/h}\;dx
			+\int_{\Cor _{\theta_0,\theta_1,\varepsilon}}  fw_h e^{-\Phi_l/h}\;dx
			\\&=(-1)^{l-1}i\frac{\Gamma(2/\beta)}{2\beta } \pare{e^{(-1)^l2i\theta_1}-e^{(-1)^l2i\theta_0}} ph^{2/\beta} +o(h^{2/\beta}),
		\end{split}
	\end{equation*}
	provided that $\lambda\in(1-\beta/2,1)$.
\end{proof}

We are now in position to prove Proposition ~\ref{prp:source2}.
\begin{proof}[Proof of Proposition~\ref{prp:source2}]
	Let $\Cor=\Cor_{\theta_0,\theta_2, R}$ with the two \Blu{sector opening angles $\omega_1=\theta_1-\theta_0$ and $\omega_2=\theta_2-\theta_1$. Let $(p_1,p_2)$ denote the essential limit} of $f$ at the corner.
	Applying Lemmas~\ref{lem:bterms}~and~\ref{lem:Asymp} to the identity \eqref{eq:fwint}, we obtain
	\[
	\int_\Cor fw^{(l)}\,dx
	=
	\Blu{(-1)^{l-1}i\frac{\Gamma(2/\beta)}{2\beta }} C_lh^{2/\beta}
	+o(h^{2/\beta}),
	\]
	where the coefficient $C_l$ with 
	\[
	C_l=
	\Bigl(e^{(-1)^l2i\theta_1}
	-e^{(-1)^l2i\theta_0}\Bigr)p_1
	+
	\Bigl(e^{(-1)^l2i\theta_2}
	-e^{(-1)^l2i\theta_1}\Bigr)p_2,
	\] 
	is obtained by summing the leading order contributions from the \Blu{two sectors}, while the boundary integral is exponentially small by Lemma~\ref{lem:bterms}. 
	Since the boundary integral is exponentially small,
	identity \eqref{eq:fwint} implies $C_l=0$ for $l=1,2$,
	which yields 
	\begin{equation}\label{eq:fsys1}
		\Pare{e^{(-1)^l2i\theta_1}-e^{(-1)^l2i\theta_0}} p_1+\Pare{e^{(-1)^l2i\theta_2}-e^{(-1)^l2i\theta_1}} p_2=0,\qquad l=1,2.
	\end{equation}
	
	Assume\Blu{, to the contrary, that $\Cor$ is an admissible $2$-sector} corner of $f$. 
	
	\Blu{First, suppose that $\omega=\pi$, as in Case~\ref{cs:pitot} of} Definition~\ref{dfn:splitcorner}. We then infer from \eqref{eq:fsys1} that
	\begin{equation*}
		\Pare{e^{(-1)^l2i\omega_1}-1} \pare{p_1- p_2}=0,
	\end{equation*}
	\Blu{Since $0<\omega_1<\pi$, we obtain $p_1=p_2$, contradicting admissibility.} 
	
	\Blu{We next consider Case~\ref{cs:fn0}, where $\omega\neq\pi$.
	If $\omega_1=\pi$, then $0<\omega_2<\pi$ and admissibility requires $p_2\neq0$.
	Since $e^{\pm2i\theta_1}=e^{\pm2i\theta_0}$, equation~\eqref{eq:fsys1} implies
	\[
	(e^{\pm2i\theta_2}-e^{\pm2i\theta_1})p_2=0.
	\]
	As $\omega_2\in(0,\pi)$, this gives $p_2=0$, a contradiction.
	Similarly, if $\omega_2=\pi$, then $0<\omega_1<\pi$, and \eqref{eq:fsys1} implies $p_1=0$, again contradicting admissibility.}
	
	\Blu{Finally, suppose that none of $\omega,\omega_1,\omega_2$ equals $\pi$. The relations \eqref{eq:fsys1} form a homogeneous linear system $A(p_1,p_2)^T=0$, where}
	\[
	A=\begin{bmatrix}
		e^{-2i\theta_1}-e^{-2i\theta_0}
		&
		e^{-2i\theta_2}-e^{-2i\theta_1}
		\\[2mm]
		e^{2i\theta_0}-e^{2i\theta_1}
		&
		e^{2i\theta_1}-e^{2i\theta_2}
	\end{bmatrix}.
	\]
	A direct calculation shows that
	\[
	\det A
	=
	-8i
	\sin(\theta_2-\theta_0)
	\sin\omega_1
	\sin\omega_2.
	\]	
	Since none of
	$\theta_2-\theta_0$,
	$\omega_1$,
	or $\omega_2$
	equals $\pi$,
	we have $\det A\neq0$.
	Hence the system admits only the trivial solution $p_1=p_2=0$, which contradicts admissibility.
\end{proof}

\subsection{Proofs of Theorems~\ref{thm:splitscatter}~and~\ref{thm:inversescattering} and Propositions~\ref{prp:invSourceGeneral}~and~\ref{prp:inversescattering1}}
Theorems~\ref{thm:splitscatter}~and~\ref{thm:inversescattering} as well as Proposition~\ref{prp:invSourceGeneral} are consequences of Proposition \ref{prp:source2}.
\begin{proof}[Proof of Theorem \ref{thm:splitscatter}]
	We argue by contradiction. Suppose that a source $(f,\Sigma)$ with an admissible \Blu{$2$-sector} corner $\Cor $ is non-radiating. Then, by Rellich's lemma and unique continuation, the corresponding solution $u$ to the forward scattering problem \eqref{eq:sourcescattering} satisfies $u\equiv 0$ outside $\Supp_tf$. As a consequence, $u$ satisfies \eqref{eq:sourcelocprob}. However, this contradicts Proposition \ref{prp:source2}, which states that $\Cor$ is not an admissible \Blu{$2$-sector} corner of $f$.
\end{proof}
\begin{proof}[Proof of Theorem \ref{thm:inversescattering}]
	Suppose that $(f^{(j)},\Sigma_j)$, $j=1,2$, produce the same far-field. Denote $f=f^{(1)}-f^{(2)}$ and $u=u_{1}-u_{2}$, where $u_{j}$ is the solution to the forward scattering problem \eqref{eq:sourcescattering} with the source term $f^{(j)}$, $j=1,2$. Then $u$ solves \eqref{eq:sourcescattering} for the source $f$ with the far-field $u_\infty\equiv 0$. By Rellich's lemma and unique continuation, we have that $u\equiv 0$ in the exterior of $\tS{f}$. 
	
	Assume, to the contrary, that $\Conv (\Sigma_1)\neq\Conv (\Sigma_2)$. Then we can assume without loss of generality that $\Conv (\Sigma_1)\not\subseteq\Conv (\Sigma_2)$.
	Since both convex hulls are convex polygons, there exists a vertex of $\Conv(\Sigma_1)$ lying outside $\Conv(\Sigma_2)$; Otherwise $\Conv (\Sigma_1)\subseteq\Conv (\Sigma_2)$ by convexity.
	
	Let $x'$ be a vertex of $\Conv (\Sigma_1)$ with $x'\notin\Conv (\Sigma_2)$. Then $f=f^{(1)}$ in $B_R(x')$ for some $R>0$. Notice that $x'$ is a vertex of $\Sigma_1$, which connects to infinity in the exterior of $\Sigma_1\cup\Sigma_2\supseteq\tS f$. Hence  $x'\in\partial\tS{f}$. In particular, $u$ satisfies \eqref{eq:sourcelocprob} on $\Cor $, which is the corner of $\Sigma_1$, as well as $\tS{f}$, at the vertex $x'$. By Proposition~\ref{prp:source2}, we arrive at a contradiction that $\Cor $ is an admissible \Blu{$2$-sector} corner of $f=f^{(1)}$. 
\end{proof}
\begin{proof}[Proof of Proposition~\ref{prp:invSourceGeneral}]
	We need only to prove the result for $\Gamma_1$ in the case when $\partial\Sigma_1\backslash\Sigma_2\neq \emptyset$.
	Let $D$ be a connected component of $\Sigma_1^\circ\backslash\Sigma_2$ such that $\Gamma_1\subseteq \partial D$. 
	By Rellich's lemma  and unique continuation we obtain that  
	\begin{equation*}
		\begin{cases}
			(\Delta+k^2\rho)u_1=f^{(1)}\AnD (\Delta+k^2\rho)u_2=0 &\quad\text{in $D$}, \\
			u_1=u_2 \AnD \Blu{\partial_\nu u_1=\partial_\nu u_2} &\quad\text{on $\Gamma_1$},
		\end{cases}
	\end{equation*}
	where $u_j$ is the scattering solution to \eqref{eq:sourcescattering} with $f=f^{(j)}$, $j=1,2$. The difference $u=u_1-u_2$ satisfies the local problem \eqref{eq:sourcelocprob} in $D$, with source term $f=f^{(1)}$.
	The result is then a direct consequence of Proposition~\ref{prp:source2} by considering the problem for $u:=u_1-u_2$.
\end{proof}

We proceed with the proof of Proposition~\ref{prp:inversescattering1}.
\begin{proof}[Proof of Proposition~\ref{prp:inversescattering1}]
	\Blu{Let $x'\in\partial G$ be a vertex of $\Sigma_1$.} 	Assume without loss of generality that $x'=0$. Denote $f=f^{(1)}-f^{(2)}$ and $u=u_{1}-u_{2}$, where $u_{j}$ is the solution to the forward scattering problem \eqref{eq:sourcescattering} with the source term $f^{(j)}$, $j=1,2$. 
	Then $u$ solves \eqref{eq:sourcescattering} for the source $f$ with the far-field $u_\infty\equiv 0$. \Blu{By Rellich's lemma and unique continuation, $u=0$ in $G$.
	
	We first show that $0\in\partial\Sigma_2$.
	Suppose, to the contrary, that $0\notin\partial\Sigma_2$.
	Since $0\in\partial G$, we have $0\notin\Sigma_2$. 
	As a consequence, there exist $R>0$ and a truncated sector $\Cor$ such that $B_R\cap\Sigma_1=\Cor$ and $B_R\cap\Sigma_2=\emptyset$. Moreover, $u$ satisfies \eqref{eq:sourcelocprob} in $\Cor$ with $f=f^{(1)}$. Therefore, by Proposition~\ref{prp:source2} we deduce that $\Cor$ is not an admissible $2$-sector corner of $f^{(1)}$. However, notice that $\Cor$ is an admissible $1$-sector corner of $f=f^{(1)}$. By subdividing $\Cor$ into two sectors with the same essential limit $p_1$, we may regard it as an admissible $2$-sector corner, which is a contradiction.
	
	Now we have $0\in\partial\Sigma_1\cap\partial\Sigma_2\cap\partial G$. Since $0\in\partial G$ and $\Sigma_1,\Sigma_2$ are polygonal, there is an open angular sector with vertex $0$ contained locally in $G$. We choose a common angular branch through the exterior sector, so that both local sectors are represented within an interval of length less than $2\pi$.
	Then we can find $R>0$ such that
	\[
	\Sigma_j^\circ\cap B_R=\Cor_{\theta_{j,0},\theta_{j,1},R}:=\Cor_j,
	\qquad j=1,2,
	\]
	with
	\[
	\omega_1=\theta_{1,1}-\theta_{1,0}\in(0,\pi)\cup(\pi,2\pi)\AND \omega_2=\theta_{2,1}-\theta_{2,0}\in(0,2\pi).
	\]
	After reducing $R$ if necessary, we may choose a circular wedge $\Cor$ of radius $R$ and opening angle in $(0,2\pi)$ that contains $\Cor_1\cup\Cor_2$ and whose radial sides, away from the origin, lie in $G$. In particular, $u$ satisfies \eqref{eq:sourcelocprob} in $\Cor$ with $f=f^{(1)}-f^{(2)}$.
	Similar to the proof of Proposition~\ref{prp:source2}, we can apply the integral identity \eqref{eq:fwint} on $\Cor$ with the CGO solutions $w^{(l)}$. While the boundary integral decays exponentially as $h\to 0$, the volume integral can be decomposed into 
	\begin{equation*}
		\int_{\Cor_1} f^{(1)}w^{(l)}\,dx - \int_{\Cor_2} f^{(2)}w^{(l)}\,dx.
	\end{equation*} 
}
	Similar to \eqref{eq:fsys1}, we derive that
	\begin{equation*}
		\Pare{e^{(-1)^l2i\theta_{1,1}}-e^{(-1)^l2i\theta_{1,0}}} p_1\Blu{-}\Pare{e^{(-1)^l2i\theta_{2,1}}-e^{(-1)^l2i\theta_{2,0}}} p_2=0,\qquad l=1,2.
	\end{equation*}
	where $p_j$ is the \Blu{essential limit of $f^{(j)}$ at $0$}, $j=1,2$.
	The relation \eqref{eq:inversecornercompat} is then verified by observing that
	\begin{equation*}
		e^{(-1)^l2i\theta_{j,1}} - e^{(-1)^l2i\theta_{j,0}}	= (-1)^{l}2ie^{(-1)^{l}i(\theta_{j,1}+\theta_{j,0})}\sin(\theta_{j,1}-\theta_{j,0})\Blu{=(-1)^{l}2ie^{(-1)^{l}2i\gamma_j}\sin\omega_j}.
	\end{equation*}

	\Blu{Finally, recalling that $\sin\omega_1\neq0$ and $p_j\neq 0$, $j=1,2$, then \eqref{eq:inversecornercompat} implies $\sin\omega_2\neq0$ and $e^{4i\gamma_1} = e^{4i\gamma_2}$.
	In other words, $0$ is also a vertex of $\Sigma_2$ and $\gamma_1=\gamma_2+m\pi/2$ for some $m\in\mathbb{Z}$.
	Applying \eqref{eq:inversecornercompat} again we obtain $p_1\sin\omega_1 =(-1)^{m}p_2\sin\omega_2$.}
\end{proof}

\section{Medium Scattering}\label{sec:medium}
In this section, we consider scattering due to the presence of a penetrable inhomogeneous medium. The forward scattering problem can be formulated as
\begin{equation}\label{eq:mediumscattering}
	\begin{cases}
		\Delta u+k^2\rho u=0 \qquad\text{in $\RR^2$},& \\
		\lim_{r\to\infty} r^{1/2}\left(\frac{\partial u_\text{sc}}{\partial r}-ik u_\text{sc}\right)=0, &
	\end{cases}
\end{equation}
where $\uin=u-\usc$ is a given incident field that satisfies the Helmholtz equation $\Delta \uin+k^2\uin=0$ in $\RR^2$, and the real-valued term $\rho\in L^\infty(\RR^2)$ describes the inhomogeneous medium. 

As in the source scattering setting, we denote by the pair $(\rho-1,\Sigma)$ the contrast $\rho-1$ along with its compact total support $\Sigma:=\tS{(\rho-1)}$, which describes the shape of the medium. We also assume that $\overline{\Sigma^\circ}=\Sigma$.
We say that a medium $\rho$, or $(\rho-1,\Sigma)$, has an (admissible) \Blu{$n$-sector} corner if $f:=\rho-1$ has an (admissible) \Blu{$n$-sector} corner by Definition~\ref{dfn:splitcorner}. 
In this case, \Blu{the essential limits of $\rho$ at the corner are $(1+p_1,\ldots,1+p_n)$, where $(p_1,\ldots,p_n)$ are the essential limits of $\rho-1$.}

The analysis follows a similar local strategy as in Section~\ref{sec:source}, but the interaction between the incident wave and the medium introduces additional compatibility conditions depending on the vanishing order of the incident field. We first establish the direct scattering results and then derive their inverse consequences.

\subsection{Main Results}
The following result states that incident waves with a non-zero value at the corner are always scattered by media that have admissible \Blu{$2$-sector} corners:
\begin{thm}\label{thm:splitobscatter}
     Any medium that has an admissible \Blu{$2$-sector} corner always scatters incident waves that are not zero at the corner tip.
\end{thm}
Theorem~\ref{thm:splitobscatter} concerns incident waves that are nonvanishing at the corner tip. For higher order vanishing, the situation is considerably more delicate. Instead of unconditional corner scattering, one obtains explicit compatibility conditions relating the split-corner geometry, the \Blu{essential limits} of the contrast, and the vanishing order of the incident wave.
\begin{prp}
	\label{prp:splitcorner}
	Let $(u,v)\in H^1(\Cor )\times H^1(\Cor )$ satisfy
	\begin{equation}\label{eq:inttranscorn}
		\begin{split}
			\Delta u+k^2\rho u=0,\quad \Delta v+k^2 v=0\quad &\text{in $\Cor $},\\
			u=v,\quad \partial_\nu u=\partial_\nu v\quad&\text{on $\partial \Cor \cap B_R$}.
		\end{split}
	\end{equation}
	Assume that \Blu{$v$ is not identically zero and extends as a solution to $\Delta v+k^2 v=0$ in a neighborhood of the corner with $v=0$ at the corner tip. If $\Cor$ is a $2$-sector corner of $\rho$, then
	\begin{equation}\label{eq:Theta}
		|P_1|=|P_{N+1}| \AND P_1a_1+e^{-2iN\gamma}P_{N+1}a_2=0,
	\end{equation}
	where $N$ is the vanishing order of $v$ at the corner tip, $a_1$ and $a_2$ are the coefficients in the first nonzero term $v_N=(a_1e^{iN\theta}+a_2e^{-iN\theta})r^N$ of the Taylor expansion of $v$, $(p_1,p_2)$ is the \Blu{essential limit} of $\rho-1$ at $\Cor=\Cor_{\theta_0,\theta_2,R}$, $\gamma:=\theta_0+\omega_1/2$ is the bisector angle of the first sector, $\omega_j$ is the opening angle of the $j$-th sector, $\omega=\omega_1+\omega_2$, and
	\[		
	P_n:=\frac{1}{n}\pare{p_1\sin(n\omega_1)+p_2e^{-in\omega}\sin(n\omega_2)}.
	\]
	Moreover, if $\Cor$ is admissible, then $|a_1|=|a_2|$ and $p_1\neq p_2$.}
\end{prp}
Note in Proposition~\ref{prp:splitcorner} that we use the standard result that $v_N$, as the first non-zero term in the Taylor series of a solution $v$ to the Helmholtz equation, is harmonic.
\begin{rem}
	Proposition~\ref{prp:splitcorner} \Blu{gives strong necessary conditions for non-scattering when the incident wave vanishes at the corner tip. In particular, at an admissible $2$-sector corner, the coefficients of the first nonzero term of the incident wave must satisfy $|a_1|=|a_2|$, in addition to the compatibility conditions in \eqref{eq:Theta}. The proposition also} recovers the classical corner scattering result for admissible \Blu{$1$-sector} corners.
	
	Figure~\ref{fig:bigtheta01} illustrates the exceptional geometries satisfying \Blu{the first condition in} \eqref{eq:Theta} for several choices of $(p_1,p_2)$ when $N=1$. For each fixed $(p_1,p_2)$, these configurations form a one-dimensional subset of the two-dimensional parameter space $(\omega_1,\omega_2)$, \Blu{illustrating the restrictive nature of the compatibility condition.}. 
	
	\Blu{For large vanishing order $N$, the first condition in \eqref{eq:Theta} becomes more restrictive. Indeed, 
	\[
		\left|p_1\sin\omega_1+p_2e^{-i\omega}\sin\omega_2\right|
		\leq \frac{|p_1|+|p_2|}{N+1}.
	\]
	Thus, when $N$ is large, the angular expression on the left must be small relative to the essential limits. In particular, the compatibility condition approaches $p_1\sin\omega_1+p_2e^{-i\omega}\sin\omega_2=0$, which characterizes the non-admissible configurations.}
\end{rem}
\begin{figure}[h!]
	\centering
	\includegraphics[width=0.9\textwidth]{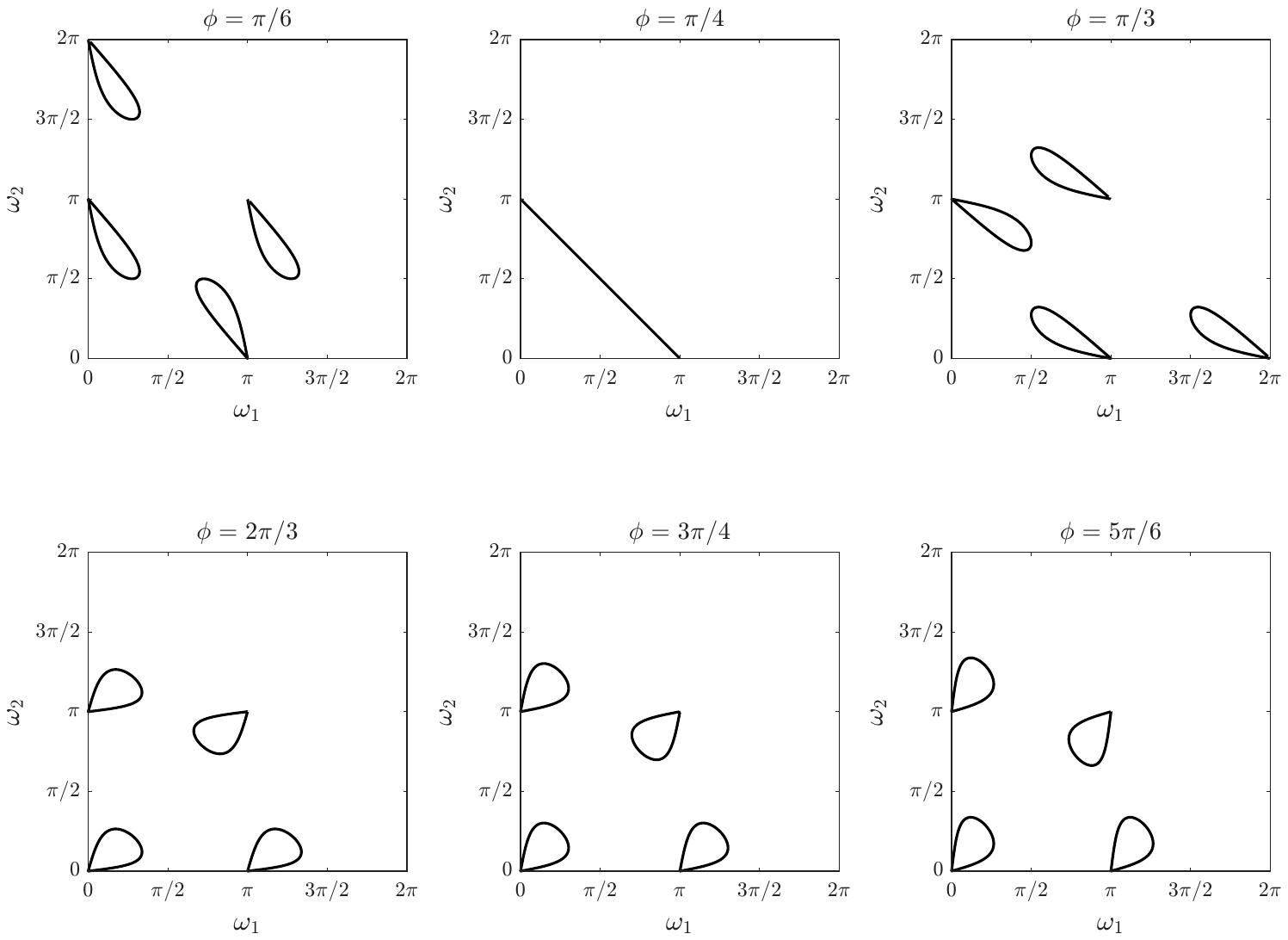}
	\caption{Plots of $\omega_1,\omega_2$ satisfying \Blu{the first condition in} \eqref{eq:Theta} with $N=1$, where $(p_1,p_2)=(\cos \phi,\sin \phi)p_0$ is assumed. Note for the case when $\phi=\pi/4$, namely if $p_1=p_2$, we must have $\omega=\omega_1+\omega_2=\pi$, which is not an admissible case by Definition~\ref{dfn:splitcorner}.}
    \label{fig:bigtheta01}
\end{figure}

The following class of \Blu{$2$-sector} corners always scatters every incident wave. 
\begin{dfn}\label{dfn:splitcorner2special}
	Let $(\rho,\Sigma)$ be a medium with a \Blu{$2$-sector corner $\Cor$, with sector
	opening angles $\omega_1,\omega_2$ and \Blu{essential limits} $(p_1,p_2)$ of $\rho-1$. The corner $\Cor$} is called \textit{degenerate} if at least one of the following is satisfied.
	\begin{enumerate}[(i)]
		\item\label{cs:1split} $p_1=p_2\neq 0$ and $\omega:=\omega_1+\omega_2\neq\pi$, \Blu{so that $\Cor$ can also be viewed as an admissible $1$-sector} corner of $\rho$.
		\item\label{cs:flat} $p_1\neq p_2$ and $\omega=\pi$.
		\item\label{cs:partflat} \Blu{Either $\omega_1=\pi$ and $p_2\neq0$, or $\omega_2=\pi$ and $p_1\neq0$.}
	\end{enumerate}
\end{dfn}
\begin{thm}\label{thm:2splitscatter}
	Any medium possessing a degenerate \Blu{$2$-sector} corner always scatters any incident wave.
\end{thm}

The local compatibility analysis also yields the following inverse uniqueness result.
\begin{prp}\label{prp:invMedGeneral}
	Suppose that $(\rho^{(j)},\Sigma_j)$, $j=1,2$, are two media producing the same far-field from a single fixed incident wave. Then for each $j=1,2$, $(\rho^{(j)},\Sigma_j)$ does not have any degenerate \Blu{$2$-sector} corner on $\Gamma_j$, where $\Gamma_j$ is any connected and relatively-open subset of $\partial\Sigma_j\backslash\Sigma_{j+1}$ that connects to infinity in $\RR^2\backslash(\Sigma_1\cup \Sigma_2)$, with $\Sigma_3:=\Sigma_1$.
\end{prp}

As in the source scattering case, the local analysis immediately yields uniqueness of the polygonal convex hull. 
\begin{thm}[Unique determination of polygonal convex hulls from a single incident wave]\label{thm:inversemedscattering}
	Let $(\rho,\Sigma)$ be a medium such that $\Sigma$ is a union of disjoint (possibly concave) polygons and each corner of $\Sigma$ is \Blu{an admissible $1$-sector corner of} $\rho$. Then $\Conv (\Sigma)$ can be uniquely recovered from the corresponding far-field $u_\infty$ for a fixed wavenumber $k$ and incident field $\uin$.
\end{thm}

\subsection{Proof of the main results for medium scattering and inverse scattering}\label{ssec:obstscatterthms}
Theorem~\ref{thm:splitobscatter}, which concerns incident waves that do not vanish at \Blu{an admissible} corner tip, is a direct consequence of Proposition~\ref{prp:source2}.
\begin{proof}[Proof of Theorem~\ref{thm:splitobscatter}]
	Let $(\rho-1,\Sigma)$ be a given medium with an admissible \Blu{$2$-sector} corner $\Cor $. We assume, to the contrary, that there exists an incident wave $v=\uin$ of vanishing order 0 that is not scattered by $\rho$. Then the solution $u\in H^1_{\text{loc}}(\RR^2)$ to the forward scattering problem \eqref{eq:mediumscattering} induced by $v$ is such that the pair $(u,v)\in H^1(\Cor )\times H^1(\Cor )$ solves the problem \eqref{eq:inttranscorn} in the corner $\Cor $, and $v$ solves $(\Delta+k^2)v=0$ in all of $\mathbb{R}^2$. As a consequence, $w:=u-v$ satisfies \eqref{eq:sourcelocprob} with $f:=-k^2(\rho-1)v$. 
	
	Since $v$ is analytic and nonvanishing at the corner tip, multiplication by $v$ preserves \Blu{the sector decomposition and multiplies each essential limit by the same nonzero factor $-k^2v(0)$. Since $v(0)\neq0$, at least one of the real and imaginary parts of $-k^2v(0)$ is nonzero. The corresponding real or imaginary part of $f$ therefore has the same admissible $2$-sector corner.} Consequently, $f=-k^2(\rho-1)v$ also possesses an admissible \Blu{$2$-sector} corner. However, this contradicts Proposition~\ref{prp:source2}.
\end{proof}

Next we prove Theorem~\ref{thm:2splitscatter}, which is a consequence of Proposition~\ref{prp:splitcorner}.
\begin{proof}[Proof of Theorem~\ref{thm:2splitscatter}]
	Let $(\rho,\Sigma)$ be a medium with a degenerate \Blu{$2$-sector} corner $\Cor=\Cor_{\theta_0,\theta_2, R}$, with $\omega_j=\theta_j-\theta_{j-1}$ the \Blu{opening angle of the $j$-th sector, $j=1,2$, and $(p_1,p_2)$ the essential limit} of $\rho-1$ at the corner. 
	Assume the contrary that there \Blu{exists} an incident wave $v$ that is not scattered by $(\rho,\Sigma)$.
	Recall that any degenerate \Blu{$2$-sector} corner is also admissible. Then for incident waves that do not vanish at the corner tip, the result is a consequence of Theorem~\ref{thm:splitobscatter}.
	In the rest of the proof, we assume that the vanishing order of the incident wave is $N\in\NN_+$.
	
	Suppose $p_1=p_2$ and $\omega:=\omega_1+\omega_2\neq\pi$. Then $\Cor$ \Blu{can also be viewed as an admissible $1$-sector} corner of $\rho$. It is known in this case that all incident waves will be scattered nontrivially; see for example, \cite{Xiao_2022}. 
	If $p_1\neq p_2$ and $\omega=\pi$, \eqref{eq:Theta} then \Blu{implies
	\begin{equation}\label{eq:prf_sin}
		|\sin\omega_1|=\frac{1}{N+1}|\sin(N+1)\omega_1|.
	\end{equation}
	Note for every $N\in\NN_+$ that $|\sin (N+1)t|\le (N+1)|\sin t|$, where the equal sign holds if and only if $t=m\pi$ with $m\in\mathbb{Z}$. On the other hand we have $0<\omega_1<\omega=\pi$, which contradicts \eqref{eq:prf_sin}.}
	We are left with Case~\ref{cs:partflat} in Definition~\ref{dfn:splitcorner2special}. Assume without loss of generality that $p_1\neq0$ and $\omega_2=\pi$. In this case \Blu{the first condition in} \eqref{eq:Theta} is equivalent to \eqref{eq:prf_sin}, which \Blu{implies that $\omega_1=\pi$ as $\omega\in(0,2\pi)$. However, in this case we have $\omega=\omega_1+\omega_2=2\pi$, which contradicts the definition of a corner.}
\end{proof}

Next we present the proof of Proposition~\ref{prp:invMedGeneral}.
\begin{proof}[Proof of Proposition~\ref{prp:invMedGeneral}]
	Denote by $G$ the unbounded connected component of $\RR^2\backslash (\Sigma_1\cup \Sigma_2)$. Then Rellich's lemma implies that $u_1=u_2$ in $G$, where $u_j$ is the solution to the forward scattering problem \eqref{eq:mediumscattering} for the medium $\rho_j$ and incident wave $\uin$.
	
	Let $\Gamma_1$ be a connected and relatively-open subset of $\partial\Sigma_1\backslash\Sigma_{2}$ that connects to infinity in $\RR^2\backslash(\Sigma_1\cup \Sigma_2)$, and suppose there is a (possibly flat) corner $\Cor$ of $(\rho^{(1)},\Sigma_1)$ on $\Gamma_1$. Then there exists $R>0$ such that $(u,v):=(u_1,u_2)$ satisfies \eqref{eq:inttranscorn}. If $\rho^{(1)}$ has a degenerate \Blu{$2$-sector} corner on $\Gamma_1$, we can then follow the proof of Theorem~\ref{thm:2splitscatter} and obtain a contradiction.
\end{proof}

Theorem \ref{thm:inversemedscattering} is in fact a direct consequence of Proposition~\ref{prp:invMedGeneral}.
\begin{proof}[Proof of Theorem \ref{thm:inversemedscattering}]
	Let $(\rho_j,\Sigma_j)$, $j=1,2$, be two media producing the same far-field from a fixed incident wave $\uin$, where $\Sigma_j$ is a union of disjoint (possibly concave) polygons, and suppose that each corner of $\Sigma_j$ is \Blu{an admissible $1$-sector corner of} $\rho_j$. Assume for the sake of contradiction that $\Conv (\Sigma_1)\neq\Conv (\Sigma_2)$. 
	Then, up to a swap of indices, there exists a corner $\Cor$ of $\rho_1$ whose corner tip lies on $\partial\Sigma_1\backslash\Sigma_2$ and connects to infinity in $\RR^2\backslash(\Sigma_1\cup\Sigma_2)$. Moreover, $\Cor$ is an admissible \Blu{$1$-sector} corner of $\rho_1$, which contradicts the conclusion of Proposition~\ref{prp:invMedGeneral}. 
\end{proof}

The proof of Proposition~\ref{prp:splitcorner} follows a similar strategy as Proposition~\ref{prp:source2}. 
Applying the integral identity \eqref{eq:wavescatterint}, together with CGO solutions and the local Taylor expansion of the incident field, yields the required compatibility conditions.

Given $(u,v)$ satisfying \eqref{eq:inttranscorn}, we have for any solution $w\in H^1(\Cor)$ to $(\Delta+k^2\rho )w=0$ in $\Cor $ that
\begin{equation}\label{eq:wavescatterint}
    \begin{split}
    	k^2\int_{\Cor } (\rho-1) v w\;dx
    &=\int_{\partial \Cor } w\partial_\nu v -v\partial_\nu w\;dS \\
    &= \int_{\partial \Cor \cap B_R}w\partial_\nu (v-u)- (v-u)\partial_\nu w\;dS .
    \end{split}
\end{equation}
Similar to the proof of Proposition~\ref{prp:source2}, we make use of the CGO solutions $\wl$ to $\Delta w+k^2\rho  w=0$ established in Lemma \ref{lem:CGO} with a parameter $h\ll 1$, and consider the asymptotic form of \eqref{eq:wavescatterint} as $h\to0$.
With the same arguments used in the proof of Lemma~\ref{lem:bterms} we obtain
\begin{equation}\label{eq:wavescatterbint}
	\int_{\partial \Cor \cap \partial B_R}\wl\partial_\nu (v-u)- (v-u)\partial_\nu \wl\;dS=O(h^{-1}e^{-\frac{1}{h} R^\beta \delta}).
\end{equation}

Regarding the asymptotics of the LHS in \eqref{eq:wavescatterint}, we have the following result. 
\begin{lem}\label{lem:Asymprho}
	Let \Blu{$\rho\in L^\infty(\RR^2)$ satisfy $|\rho(x)-\rho_0|\leq C|x|^\alpha$ a.e. in $\Cor _{\theta_1,\theta_2,R}$ for some constants $\rho_0$, $C\geq0$, and} $\alpha\in(0,1)$, and let $v$ be a nontrivial solution to $\Delta v+k^2v=0$ in $B_R$ of vanishing order $N\in\mathbb{N}$. 
	Then, for some $\varepsilon\in\RR_+$ independent of $h$,
	\begin{equation*}
		\int_{\Cor_{\theta_{1},\theta_{2}},\varepsilon}(\rho-1)v\wl dx=\frac{\Gamma((N+2)/\beta)}{i\beta}(\rho_0-1)B_{N,l}h^{(N+2)/\beta}+o(h^{(N+2)/\beta})
		\qquad \mbox{as $h\to 0$},
	\end{equation*}
	where  $B_{N,1}=a_1 c_{0} +a_2c_{N}$ and $B_{N,2}=a_1 \overline{c_{N}} +a_2\overline{c_{0}}$ with 
	\begin{equation}\label{eq:CNl}
		c_{n}=\frac{e^{-i2(n+1)\theta_2}-e^{-i2(n+1)\theta_1}}{-2(n+1)},
	\end{equation}
	and $a_l$ are the coefficients in the first nonzero ($N$th-) term $v_N=(a_1e^{iN\theta}+a_2e^{-iN\theta})r^N$ in the Taylor expansion of $v$. 
\end{lem}
\begin{proof}[Proof of Lemma~\ref{lem:Asymprho}]
	We separate the integral into four parts.
	First, 	applying \eqref{eq:Estr^a} we obtain that
	\begin{equation*}
		\int_{\Cor _{\theta_1,\theta_2,\varepsilon}}v_{N}|\rho-\rho_0| e^{-\Phi_l/h}\;dx
		\le C h^{(2+N+\alpha)/\beta},
	\end{equation*}
	and
	\begin{equation*}
		\int_{\Cor _{\theta_1,\theta_2,\varepsilon}}\rho|v-v_{N}| e^{-\Phi_l/h}\;dx
		\le C h^{(3+N)/\beta}.
	\end{equation*}
	Second, similar to \eqref{eq:whterm} we can deduce for any $\lambda\in(0,1)$ that
	\begin{equation*}
		\begin{split}
			&\left|\int_{\Cor _{\theta_1,\theta_2,\varepsilon}} (\rho-1) v w_h e^{-\Phi_l/h}\;dx \right| \\
			&\qquad\leq \|\rho-1\|_{L^\infty}\|w_h\|_{L^{\frac{1}{1-\lambda}}(\Cor _{\theta_1,\theta_2,\varepsilon})}\PAre{\int_{\Cor _{\theta_1,\theta_2,\varepsilon}} |x|^{N/\lambda} e^{-\delta |x|^\beta/(\lambda h)}\;dx}^{\lambda} \\
			&\qquad\leq C h^{1+N/\beta+2\lambda/\beta},
		\end{split}
	\end{equation*}
	which is a term of order $o(h^{(N+2)/\beta})$ so long as we take $\lambda\in(1-\beta/2,1)$.
	Combining the above three estimates, we conclude that all error terms are of order $o(h^{(N+2)/\beta})$. Hence
	\begin{equation}\label{eq:intoh}
		\int_{\Cor_{\theta_{1},\theta_{2}},\varepsilon}(\rho-1)v\wl - (\rho_0-1)v_{N}e^{-\Phi_l/h} dx =o(h^{(N+2)/\beta}).
	\end{equation}

	We are left with the integral term concerning $(\rho_0-1)v_{N}e^{-\Phi_l/h}$. Similar to \eqref{eq:leading}, applying Lemma~\ref{gammest} we can calculate (see also \cite[Lemma 4.6]{Xiao_2022})
	\begin{equation}\label{eq:harmhompoly}
		\int_{\Cor _{\theta_1,\theta_2,\varepsilon}} r^Ne^{\pm iN\theta} e^{-\Phi_l/h}\;dx = -i\frac{\Gamma((N+2)/\beta)}{\beta}C_{N,l,\pm}h^{(N+2)/\beta}+o(h^{(N+2)/\beta}),
	\end{equation}
		where 
		\begin{equation*}
			C_{N,l,\pm}=\frac{ e^{i[((-1)^l\pm 1)N+(-1)^l 2]\theta_2}-e^{i[((-1)^l\pm 1)N+(-1)^l 2]\theta_1}}{((-1)^l\pm 1)N+(-1)^l 2}.
		\end{equation*}
	The proof can be then concluded by combining \eqref{eq:intoh}, \eqref{eq:harmhompoly}, and the expression $v_N=(a_1e^{iN\theta}+a_2e^{-iN\theta})r^N$.
\end{proof}

Now that we have the asymptotics of \eqref{eq:wavescatterint}, we are ready to prove Proposition \ref{prp:splitcorner}:
\begin{proof}[Proof of Proposition~\ref{prp:splitcorner}]
	Let $\varepsilon\in(0,R)$ be such that $v$ extends analytically as a solution to $\Delta v+k^2 v=0$ in $B_\varepsilon(0)$ \Blu{and $|\rho-1-p_j|=O(|x|^{\alpha_j})$ a.e. in ${\Cor }_{\theta_{j-1},\theta_{j},\varepsilon}$} for $j=1,2$.  
	Applying Lemma~\ref{lem:Asymprho} together with the boundary estimate \eqref{eq:wavescatterbint}, identity \eqref{eq:wavescatterint} becomes
	\begin{equation*}
		\int_{\Cor } (\rho-1) v w^{(l)}\;dx=\Blu{\frac{\Gamma((N+2)/\beta)}{i\beta}}\pare{p_1B_{N,l,1}+p_2B_{N,l,2}} h^{(N+2)/\beta}+o(h^{(N+2)/\beta}),
	\end{equation*}
	\Blu{where} $B_{N,1,j}=a_1 c_{0,j} +a_2c_{N,j}$ and $B_{N,2,j}=a_1 \overline{c_{N,j}} +a_2\overline{c_{0,j}}$, $c_{n,2}=c_{n}$ defined in \eqref{eq:CNl}, and $c_{n,1}$ is the same but replacing $(\theta_2,\theta_1)$ with $(\theta_1,\theta_0)$.
	The left-hand side is exponentially small by \eqref{eq:wavescatterbint}, whereas Lemma~\ref{lem:Asymprho} shows that it admits a polynomial asymptotic expansion.
	Therefore the leading coefficient of 	$h^{(N+2)/\beta}$ must vanish.
	In other words,
	\begin{equation*}
		p_1B_{N,l,1}+p_2B_{N,l,2}=0,\qquad l=1,2.
\end{equation*}
	Equivalently, we obtain the linear system
	\Blu{
	\begin{equation}\label{eq:linsys}
		\begin{pmatrix}
			A&B\\
			\overline B&\overline A
		\end{pmatrix}
		\begin{pmatrix}a_1\\a_2\end{pmatrix}=0,
	\end{equation}
	where $A:=p_1c_{0,1}+p_2c_{0,2}$ and $B:=p_1c_{N,1}+p_2c_{N,2}$.}
	Since $(a_1,a_2)\neq(0,0)$, the coefficient matrix of \eqref{eq:linsys} must be singular, \Blu{and hence $|A|^2-|B|^2=0$}.
	Substituting the expressions for $c_{n,j}$, \Blu{we obtain $A=ie^{-2i\gamma}P_1$ and $B=ie^{-2i(N+1)\gamma}P_{N+1}$, which proves \eqref{eq:Theta}.
		
	We next assume that $\Cor$ is admissible. We first show that $A\neq0$. Suppose instead that $A=0$. Since $|A|=|B|$, we also have $B=0$. In particular, by the definition of $c_{0,j}$ we obtain
	\[
		p_1\sin\omega_1+ p_2e^{-i\omega}\sin\omega_2=0.
	\]
	If $\omega\neq\pi$, taking imaginary parts gives $p_2\sin\omega_2=0$, and consequently $p_1\sin\omega_1=0$. Since $0<\omega_j<2\pi$ and $\omega_1+\omega_2<2\pi$, at most one of $\omega_1$ and $\omega_2$ can equal $\pi$. If neither equals $\pi$, then $p_1=p_2=0$, which is incompatible with admissibility.
	Thus either $\omega_1=\pi$ and $p_2=0$, or $\omega_2=\pi$ and $p_1=0$. Both possibilities, however, contradict the corresponding admissibility condition.
	If instead $\omega=\pi$, then $\omega_2=\pi-\omega_1$, and $A=0$ gives $(p_1-p_2)\sin\omega_1=0$. Since $0<\omega_1<\pi$, we obtain $p_1=p_2$, again contradicting admissibility.
	Therefore $A\neq0$. Since $|A|=|B|$, we also have $B\neq0$. Recall from \eqref{eq:linsys} that $a_1A+a_2B=0$. Together with $|A|=|B|$, this yields $|a_1|=|a_2|$.
		
	Finally, we show that $p_1\neq p_2$. Otherwise we have from the admissibility that $p_1=p_2=p\neq0$.  Then
	\[
		A=ipe^{-i(\theta_0+\theta_2)}\sin\omega
		\AND
		B=\frac{ip}{N+1}e^{-i(N+1)(\theta_0+\theta_2)} \sin((N+1)\omega).
	\]
	The relation $|A|=|B|$ therefore becomes $(N+1)|\sin\omega|=|\sin((N+1)\omega)|$.
	Recall that $|\sin((N+1)t)|\leq (N+1)|\sin t|$, with equality if and only if $t\in\pi\mathbb Z$. Since $0<\omega<2\pi$, we conclude that $\omega=\pi$. But admissibility in this case requires $p_1\neq p_2$, a contradiction.
}
\end{proof}

\paragraph*{Acknowledgments}
The research of the authors was partially supported by the NSF grant DMS-2307737.

\bibliographystyle{abbrv}
\bibliography{splitcorner.bib}
\end{document}